\documentclass[12pt,a4paper]{article}
\usepackage{amssymb,times1,amsmath,amsfonts,amsthm,euscript,bbm}
\usepackage{hyperref}
\usepackage{mysec}
\usepackage{color}
\usepackage[normalem]{ulem}
\usepackage{cancel}

\makeatletter
\renewcommand{\@listI}%
{\leftmargin=\parindent
\partopsep=0pt
\parskip=-5pt
\topsep=3pt
\itemsep=3pt
\labelwidth=\leftmargini }
\makeatother

\makeatletter \@addtoreset {equation}{section}

\makeatother

\newtheorem{theorem}{Theorem}[section]
\newtheorem{lemma}[theorem]{Lemma}
\newtheorem{corollary}[theorem]{Corollary}

\theoremstyle{remark}
\newtheorem{example}{Example}[section]
\newtheorem{remark}{Remark}[section]
\newtheorem{assumption}{Assumption}[section]
\theoremstyle{definition}

\newcommand{\const}{{\mathop{\rm const}\nolimits}}
\newcommand{\myp}{\mbox{$\:\!$}}
\newcommand{\mypp}{\mbox{$\;\!$}}
\newcommand{\myn}{\mbox{$\;\!\!$}}
\newcommand{\mynn}{\mbox{$\:\!\!$}}

\newcommand{\PP}{P}
\newcommand{\QQ}{Q}
\newcommand{\EE}{E}
\newcommand{\Var}{\mathrm{Var}}
\newcommand{\Cov}{\mathrm{Cov}}
\renewcommand{\det}{\mathop{\rm det}\nolimits}
\newcommand{\supp}{\mathop{\rm supp}\nolimits}
\newcommand{\bfOne}{\mathbbm{1}}

\newcommand{\NN}{\mathbb{N}}
\newcommand{\ZZ}{\mathbb{Z}}
\newcommand{\RR}{\mathbb{R}}
\newcommand{\CC}{{\mathbb C}}

\newcommand{\calH}{{\mathcal H}}

\newcommand{\CP}{\varPi}

\newcommand{\calX}{\mathcal{X}}

\newcommand{\rme}{{\mathrm e}}
\newcommand{\rmi}{{\mathrm i}}
\newcommand{\dif}{{\mathrm d}}
\newcommand{\D}{D}
\newcommand{\topp}{{\mbox{\tiny$\top$}}}

\begin{document}
\begin{center}
{\LARGE Limit shape of random
convex polygonal %\\[.4pc]
lines:\\
Even more universality

}

\vspace{1.1pc} {\large Leonid V.\ Bogachev}
%\thanks{Research supported in part by a Leverhulme Research Fellowship.}
\\[1pc]
\textrm{Department of Statistics, School of Mathematics, University
of Leeds, Leeds LS2 9JT, UK.}\\
\textrm{E-mail:} {\tt L.V.Bogachev@leeds.ac.uk}
\end{center}

\vspace{.5pc} \hfill \textit{To the memory of\/
Yu.\,V.~Prokhorov}\hspace{1pc}

%\dedicatory{To the memory of Yuri Vasilyevich Prokhorov (15 December
%1929 -– 16 July 2013)}

\noindent
\begin{abstract}
The paper concerns the limit shape (under some probability measure)
of convex polygonal lines with vertices on $\ZZ_+^2$\myp, starting
at the origin and with the right endpoint $n=(n_1,n_2)\to\infty$. In
the case of the uniform measure, an explicit limit shape
$\gamma^*:=\{(x_1,x_2)\in\RR_+^2\myp\colon
\sqrt{1-x_1}+\sqrt{x_2}=1\}$ was found independently by Vershik
(1994),
% [A.M.~Vershik, The limit shape of convex lattice polygons
%and related topics, Funct.\ Anal.\ Appl.\ 28 (1994) 13--20],
B\'ar\'any (1995),
%[I.~B\'ar\'any, The limit shape of convex lattice
%polygons, Discrete Comput.\ Geom.\ 13 (1995) 279--295],
and Sinai (1994).
%[Ya.G.~Sina\u{\i}, Probabilistic approach to the
%analysis of statistics for convex polygonal lines, Funct.\ Anal.\
%Appl.\ 28 (1994) 108--113].
Recently, Bogachev and Zarbaliev (2011)
%[L.V.~Bogachev, S.M.~Zarbaliev, Universality of the limit shape of convex lattice
%polygonal lines, Ann.\ Probab.\ 39 (2011) 2271--2317]
proved that the limit shape $\gamma^*$ is universal for a certain
parametric family of multiplicative probability measures
generalizing the uniform distribution. In the present work, the
universality result is extended to a much wider class of
multiplicative measures, including (but not limited to) analogs of
the three meta-types of decomposable combinatorial structures ---
multisets, selections, and assemblies. This result is in sharp
contrast with the one-dimensional case where the limit shape of
Young diagrams associated with integer partitions heavily depends on
the distributional type.


\medskip \noindent \emph{2010 MSC}:
%Mathematics Subject Classification}:
Primary 52A22; Secondary 05A17, 05D40, 60F05, 60G50
% 05A17 Partitions of integers [See also 11P81, 11P82, 11P83]
% 05D40 Probabilistic methods
%
% 52A10 Convex sets in $2$ dimensions (including convex curves) [See also 53A04]}
% 52A22 Random convex sets and integral geometry [See also 53C65, 60D05]
% 52A27 Approximation by convex sets
% 52B20 Lattice polytopes (including relations with commutative algebra and algebraic geometry)
% [See also 06A11, 13F20, 13Hxx]
%
% 53A04 Curves in Euclidean space
% 53A25 Differential line geometry
%
% 60F05 Central limit and other weak theorems
% 60G50 Sums of independent random variables; random walks
%
\end{abstract}

\tableofcontents

\vspace{1.5mm}

\normalfont

\section{Introduction}\label{sec1}

In this paper, a \textit{convex lattice polygonal line} $\varGamma$
is understood as a piecewise linear continuous path on the plane
starting at the origin $0=(0,0)$, with vertices on the
two-dimensional integer lattice $\ZZ^2$ and such that the
inclination of its consecutive edges is strictly increasing, staying
between $0$ and $\pi/2$ (clearly, any such $\varGamma$ lies within
the first coordinate quadrant). Let $\CP$ be the set of all convex
lattice polygonal lines \emph{with finitely many edges}, and denote
by $\CP_{n}\subset\CP$ the subset of polygonal lines
$\varGamma\in\CP$ with the right endpoint
$\xi_\varGamma=(\xi_1,\xi_2)$ \emph{fixed} at
$n=(n_1,n_2)\in\ZZ_+^2:=\{(k_1,k_2)\in\ZZ^2\colon k_j\ge 0\}$.

If each space $\CP_n$ is endowed with a probability measure $\PP_n$,
respectively (e.g., a uniform measure making all $\varGamma\in\CP_n$
equiprobable), then one can speak of \emph{random polygonal lines},
and it is of interest to study their asymptotic statistics as
$n\to\infty$ (say, assuming that $n_2/n_1\to c\in(0,\infty)$). In
particular, the \emph{limit shape} of random polygonal lines,
whenever it exists, is defined as a planar curve $\gamma^*$ such
that, for any $\varepsilon>0$,
\begin{equation}\label{eq:LLN}
\lim_{n\to\infty}\PP_n\{\varGamma\in\CP_n\colon
d\myp(\tilde{\varGamma}_n,\gamma^*)\le\varepsilon\}=1,
\end{equation}
where $\tilde{\varGamma}_n:=\mathfrak{s}_n(\varGamma)$, with a
suitable scaling transform $\mathfrak{s}_n\colon\RR^2\to\RR^2$, and
$d(\cdot,\cdot)$ is some metric on the path space, for instance
induced by the Hausdorff distance between compact sets,
\begin{equation}\label{eq:dH}
d_{\mathcal H}(A,B):=\max\myn \Bigl\{\max_{x\in A}\min_{y\in
B}|x-y|, \,\max_{y\in B}\min_{x\in A}|x-y|\Bigr\},
\end{equation}
where $|\,{\cdot}\,|$ is the Euclidean vector norm in $\RR^2$.

\begin{remark}
By definition, for a polygonal line $\varGamma\in\CP_n$ the vector
sum of its consecutive edges equals $n=(n_1,n_2)$; due to the
convexity property, the order of parts in the sum is determined
uniquely. Hence, any such $\varGamma$ represents a (two-dimensional)
integer partition of $n\in\ZZ_+^2$ which is \textit{strict} (i.e.,
without proportional parts; see \cite[\S\myp3]{V1}). Let us remark
that for ordinary (one-dimensional) integer partitions the limit
shape problem is set out differently, in terms of the associated
\emph{Young diagrams} \cite{V3,Bogachev,Yakubovich}.
\end{remark}

The limit shape and its very existence may depend on the family of
probability laws $\PP_n$\myp. With respect to the \emph{uniform}
distribution on $\CP_n$\myp, the problem was solved independently by
Vershik \cite{V1}, B\'ar\'any \cite{Barany} and Sinai \cite{Sinai},
who showed that under the natural scaling
\begin{equation}\label{eq:scaling}
\mathfrak{s}_n\colon(x_1,x_2)\mapsto (n_1^{-1}x_1,\myp
n_2^{-1}x_2),\qquad n=(n_1,n_2),\ \ n_1,n_2>0,
\end{equation}
and with respect to the Hausdorff metric $d_\calH$, the limit
\eqref{eq:LLN} holds with the limit shape $\gamma^*$ given by a
parabola arc
\begin{equation}\label{eq:gamma0}
\sqrt{1-x_1}+\sqrt{x_2}=1,\qquad 0\le x_1,x_2\le1.
\end{equation}

Recently, Bogachev and Zarbaliev \cite{BZ-DAN,BZ4} considered the
limit shape problem for a more general class of ``multiplicative''
measures $\{\PP_n\}$ of the form
\begin{equation}\label{eq:P-r}
\PP_n(\varGamma):=\frac{b(\varGamma)}{\displaystyle B_n}
  \myp,\qquad\varGamma\in\CP_n,
\end{equation}
with
\begin{equation}\label{eq:b-Gamma}
 b(\varGamma):=\prod_{e_i\in\varGamma} b_{k_i},\qquad B_n:={}\sum_{\varGamma\in\CP_n}
  b(\varGamma),
\end{equation}
where the product is over all edges $e_i$ of $\varGamma\in\CP_n$,
index $k_i$ equals the number of lattice points on the edge $e_i$
except its left endpoint, and $\{b_k\}$ is a given nonnegative real
sequence. Specifically, it has been proved in \cite{BZ-DAN,BZ4}
that, under the scaling \eqref{eq:scaling}, the same limit shape
\eqref{eq:gamma0} is valid for a parametric class of measures
$\PP_n=\PP_n^{(r)}$ \,($0<r<\infty)$ with the
coefficients\footnote{\mypp Note that for $r=1$ the formula
\eqref{eq:b-k-r} gives $b_k\equiv 1$, which implies that
$b(\varGamma)=1$ for any $\varGamma\in\CP_n$ and hence the measure
\eqref{eq:P-r} is reduced to the uniform distribution on the space
$\CP_n$.}
\begin{equation}\label{eq:b-k-r}
b_k=b_k^{(r)}:=\binom{r+k-1}{k}=\frac{r(r+1)\cdots(r+k-1)}{k!}\myp.
\end{equation}

This result has provided the first evidence in support of a
conjecture on the \emph{limit shape universality}, put forward
independently by Vershik \cite[Remark 2, p.~20]{V1}\footnote{\mypp
Page reference is given to the English translation of \cite{V1}.}
and Prokhorov \cite{Prokhorov}. The goal of the present paper is to
show that the limit shape $\gamma^*$ given by \eqref{eq:gamma0} is
universal in a much wider class of probability measures of the
multiplicative form \eqref{eq:P-r},~\eqref{eq:b-Gamma}. For
instance, along with the uniform measure on $\CP_n$ this class
contains the uniform measure on the subset $\check{\CP}_n\subset
\CP_n$ of polygonal lines that do not have any lattice points other
than vertices. More generally, measures covered by our method
include (but are not limited to) analogs of the three classical
meta-types of decomposable combinatorial structures
--- multisets, selections, and assemblies
\cite{ABT,AT} (see examples in Section \ref{sec6} below).

\begin{remark}
It should be stressed that our universality result is in sharp
contrast with the one-dimensional case, where the limit shape of
random Young diagrams heavily depends on the distributional type
\cite{Bogachev,EG,V3,Yakubovich}. Thus, the limit shape of (strict)
vector partitions is a relatively ``soft'' property; such a
distinction is essentially due to the different ways of
geometrization used in the two models (i.e., convex polygonal lines
vs.\ Young diagrams), resulting in similar but not identical
functionals responsible for the limit shape (cf.\
\cite[Sec.\,1.1]{Bogachev}).
\end{remark}

Let us state our result more precisely. Using the \emph{tangential
parameterization} of convex paths introduced in
\cite[\S\myp{}A.1]{BZ4}, consider the scaled polygonal line
$\tilde{\varGamma}_n=\mathfrak{s}_n(\varGamma)$
(see~\eqref{eq:scaling}) and let $\tilde\xi_n(t)$ denote the right
endpoint of the part of $\tilde\varGamma$ with the tangent slope
(where it exists) not exceeding $t\in[0,\infty]$. Similarly, the
tangential parameterization of the parabola arc $\gamma^*$
(see~\eqref{eq:gamma0}) is given by\footnote{\mypp It is easy to
check that the coordinate functions $(g^*_1(t),g^*_2(t))$ in
\eqref{eq:g*} satisfy the equation \eqref{eq:gamma0} (and therefore
parametrically define the curve $\gamma^*$) and, furthermore,
$g^{*\prime}_2(t)/g^{*\prime}_1(t)\equiv t$, so that the parameter
$t$ has the meaning of the tangent slope at the corresponding point
on the curve, as required.}
\begin{equation}\label{eq:g*}
g^*\myn(t)=\left(\frac{t^2+2t}{(1+t)^2}\myp,
\frac{t^2}{(1+t)^2}\right), \qquad 0\le t\le\infty,
\end{equation}
with $g^*(\infty):=\lim_{t\to\infty}g^*(t)=(1,1)$. Then the
\textit{tangential distance} between $\tilde{\varGamma}_n$ and
$\gamma^*$ is defined as
\begin{equation}\label{eq:d-T}
d_{\mathcal T}(\tilde{\varGamma}_n,\gamma^*):= \sup_{0\le
t\le\infty}\mynn|\myp\tilde\xi_n(t)-g^*\myn(t)|.
\end{equation}
It is known \cite[\S\myp{}A.1]{BZ4} that \emph{the Hausdorff\/
distance $d_{\mathcal H}$} (see~\eqref{eq:dH}) \emph{is dominated by
the tangential distance $d_{\mathcal T}$}.

A loose formulation of our result about the universality of the
limit shape is as follows.\footnote{\mypp For an exact statement and
its proof, see Theorem~\ref{th:LSP} in Section \ref{sec5.4} below.}
\begin{theorem}\label{th:main1}
Suppose that the family of\/ measures $\PP_n$ on the respective
spaces $\CP_n$ is defined via the multiplicative formulas
\eqref{eq:P-r}, \eqref{eq:b-Gamma} with the coefficients $\{b_k\}$
satisfying some mild technical conditions expressed in terms of\/
the power series expansion of\/ the function\/ $u\mapsto
\ln\mynn\bigl(\sum_{k=0}^\infty b_k u^k\bigr)$. Then, under the
scaling \eqref{eq:scaling}, for any\/ $\varepsilon>0$
\begin{equation*}
\lim_{n\to\infty}\PP_n\{\varGamma\in\CP_n\colon \,d_{\mathcal
T}(\tilde{\varGamma}_n,\gamma^*)\le\varepsilon\}=1.
\end{equation*}
\end{theorem}

\begin{remark}
Universality of the limit shape $\gamma^*$ has its boundaries: as
has been demonstrated by Bogachev and Zarbaliev
\cite{BZ2,BZ_inverse}, any $C^3$-smooth, strictly convex curve
$\gamma$ starting at the origin may serve as the limit shape with
respect to a suitable family of multiplicative probability measures
$\PP_n=\PP_n^{\gamma}$ on $\CP_n$.
\end{remark}

Following \cite{BZ-DAN, BZ4} our proof employs an elegant
probabilistic approach based on randomization and conditioning
(see~\cite{ABT,AT}) first used in the polygonal context by Sinai
\cite{Sinai}. The idea is to randomize the right endpoint
$\xi_\varGamma$ of the polygonal line $\varGamma$, originally fixed
at $n=(n_1,n_2)$, by introducing a probability measure $\QQ_z$ on
the space $\CP=\bigcup_n\myn\CP_n$ (conveniently depending on an
auxiliary ``free'' parameter $z=(z_1,z_2)$, \,$0<z_j<1$), such that
for each $n\in\ZZ_+^2$ the measure $\PP_n$ on $\CP_n$ is recovered
as the conditional distribution
$\PP_n(\cdot)=\QQ_z(\cdot\,|\mypp\CP_n)$. By virtue of the
multiplicativity of $\PP_n$ (see
\eqref{eq:P-r},~\eqref{eq:b-Gamma}), $\QQ_z$ may be constructed as a
\emph{product measure}, under which the coefficients $\{k_i\}$ in
\eqref{eq:b-Gamma} become independent (although not identically
distributed) random variables, so that $\xi_\varGamma$ is
represented as a sum of independent vectors. Thus the asymptotics of
the probability $\QQ_z(\CP_n)=\QQ_z\{\xi_\varGamma=n\}$, needed in
order to return from $\QQ_z$ to $\PP_n$, can be obtained by proving
the corresponding (two-dimensional) local limit theorem. Let us
point out that we find it more convenient to calibrate the parameter
$z$ from the asymptotic equation
$\EE_z(\xi_\varGamma)=n\left(1+o(1)\right)$ as $n\to\infty$, rather
than from the exact relation $\EE_z(\xi_\varGamma)=n$; however, this
necessitates obtaining a refined asymptotic bound on the error term
$\EE_z(\xi_\varGamma)-n$. Last but not least, the main technical
novelty that has allowed us to extend and enhance the argumentation
of \cite{BZ4} to a much more general setting considered here is that
we work with \emph{cumulants} rather than moments (see
Section~\ref{sec2.2}), which proves extremely efficient throughout.

\subsubsection*{Layout.}
The rest of the paper is organized as follows. In
Section~\ref{sec2}, we define the families of measures $\QQ_z$ and
$\PP_n$. In Section~\ref{sec3}, suitable values of the parameter
$z=(z_1,z_2)$ are chosen (Theorem~\ref{th:delta12}), which implies
convergence of ``expected'' polygonal lines to the limit curve
$\gamma^*$ (Theorems \ref{th:3.2} and~\ref{th:8.1.1a}). Refined
first-order moment asymptotics are obtained in Section~\ref{sec3.3}
(Theorem~\ref{th:4.1}), while higher-order moment sums are analyzed
in Section~\ref{sec4}. Most of Section~\ref{sec5} is devoted to the
proof of the local limit theorem (Theorem~\ref{th:LCLT}). Finally,
the limit shape results, with respect to both $\QQ_z$ and $\PP_n$,
are proved in Section~\ref{sec5.4} (Theorems \ref{th:LSQ}
and~\ref{th:LSP}).

\subsubsection*{Some general notation.}
We denote $\ZZ_+:=\{k\in\ZZ\colon k\ge0\}$,
\,$\ZZ_+^2:=\ZZ_+\mynn\times\ZZ_+$\myp, and similarly
$\RR_+:=\{x\in\RR\colon x\ge0\}$,
\,$\RR_+^2:=\RR_+\mynn\times\RR_+$\myp. The notation $\#(\cdot)$
stands for the number of elements in a set. The symbol $\lfloor
x\rfloor:=\max\myn\{k\in\ZZ\colon k\le x\}$ denotes the (floor)
integer part of $x\in\RR$\myp. The real part of a complex number
$s=\sigma + \rmi t\in\CC$ is denoted $\Re\myp(s)=\sigma$\strut{}.
For a (row-)vector $x=(x_1,x_2)\in \RR^2$, its Euclidean norm is
defined as $|x|:=\sqrt{x_1^2+x_2^2}$\mypp, and $\langle
x,y\rangle:=x\mypp y^{\myn\topp\!}=x_1y_1+x_2\myp y_2$ is the
corresponding inner product of vectors $x,y\in \RR^2$, where
$y^{\myn\topp\!}=\genfrac{(}{)}{0pt}{}{\,y_1\,}{\,y_2\,}$ is the
transpose of $y=(y_1,y_2)$. More generally,
$A^{\myn\topp\!}=(a_{ji})$ is the transpose of matrix $A=(a_{ij})$.
The matrix norm induced by the vector norm $|\,{\cdot}\,|$ is
defined by $\|A\|:=\sup_{|x|=1}\mynn|x A|$. For
$x=(x_1,x_2)\in\ZZ_+^2$ and $z=(z_1,z_2)\in \RR^2_+$ with
$z_1,z_2>0$, we use the multi-index notation
$z^{x}:=z_1^{x_1}z_2^{x_2}$. The gamma function is denoted
$\Gamma(s)= \int_0^\infty u^{s-1}\,\rme^{-u}\,\dif{u}$, and
$\zeta(s)=\sum_{k=1}^\infty k^{-s}$ is the Riemann zeta function.

Throughout the paper, the notation $n\to\infty$ (with
$n=(n_1,n_2)\in\ZZ^2_+$) is understood as $n_1,n_2\to\infty$ in such
a way that the ratio $n_2/n_1$ stays bounded, that is, $c_*\mynn\le
n_2/n_1\le c^*$ with some constants $0<c_*\mynn\le c^*\mynn<\infty$.
The asymptotic relation $x_n\asymp y_n$ between real-valued
sequences $\{x_n\}$ and $\{y_n\}$ ($n\in\ZZ^2_+$\myp) signifies that
$0<\liminf_{n\to\infty} x_n/y_n\le\limsup_{n\to\infty}
x_n/y_n<\infty$, whereas $x_n\sim y_n$ is a standard shorthand for
$\lim_{n\to\infty}x_n/y_n=1$. Thus, the limit $n\to\infty$ defined
above can itself be characterized via the asymptotic condition
$n_1\asymp n_2$; in particular, this implies that $n_1\asymp |n|$,
$n_2\asymp |n|$, where $|n|=\sqrt{n_1^2+n_2^2}\to\infty$.

\section{Probability measures on spaces of convex polygonal lines}
\label{sec2}

\subsection{Global measure\/ \myp$\QQ_{z}$ and conditional
measure\/ $\PP_n$}\label{sec2.1}

\subsubsection{Encoding of polygonal lines.}\label{sec2.1.1}
Let $\calX\subset\ZZ^2_+$ be the subset of integer vectors with
co-prime coordinates,
\begin{equation}\label{eq:X}
\calX:=\{x=(x_1,x_2)\in\ZZ^2_+\colon\gcd\mynn(x_1,x_2)=1\},
\end{equation}
where ``$\gcd$'' stands for \emph{greatest common divisor}. Note
that the set $\ZZ^2_+$ can be viewed as an integer cone (i.e., with
nonnegative integer multipliers) generated by $\calX$ as a base;
more precisely, $\ZZ^2_+$ is a disjoint union of the multiples of
$\calX$,
\begin{equation}\label{eq:cone}
\ZZ^2_+=\bigsqcup_{k=0}^\infty k\calX.
\end{equation}
That is, for each nonzero $y\in\ZZ^2_+$ there are unique $x\in\calX$
and $k\in\NN$ such that $y=k x$.

Let $\Phi:= (\ZZ_+)^{\calX}$ be the space of functions on $\calX$
with nonnegative integer values, and consider the subspace of
functions with \textit{finite support}\/,
$\Phi_0:=\{\nu\in\Phi\myp\colon\mypp\#(\supp\myp\nu)<\infty\}$,
where $\supp\myp\nu:=\{x\in \calX\colon\mypp\nu(x)>0\}$. It is easy
to see that the space $\Phi_0$ is in one-to-one correspondence with
the space $\CP=\bigcup_{n\in\ZZ_+^2}\myn\CP_n$ of all (finite)
convex lattice polygonal lines. Indeed, given a configuration
$\nu=\{\nu(x)\}\in\Phi_0$, each $x\in \calX$ specifies the
\emph{direction} of a potential edge, only utilized if
$x\in\supp\myp\nu$, in which case the value $\nu(x)=k>0$ specifies
the \emph{scaling factor}, altogether yielding a vector edge $kx$;
finally, assembling (a finite set of) all such edges into a
polygonal line is uniquely determined by fixation of the starting
point at the origin and the convexity property. Conversely, via the
same interpretation of vector edges it is evident, in view of the
decomposition \eqref{eq:cone}, that any finite polygonal line
$\varGamma\in\CP$ determines uniquely a finitely supported
configuration $\nu\in\Phi_0$. Let us point out that the case
$\nu(x)\equiv0$ corresponds to the ``trivial'' polygonal line
$\varGamma_\emptyset$ with no edges (and with coinciding endpoints).

Under the association $\CP\ni\varGamma\leftrightarrow \nu\in\Phi_0$
described above, the vector
\begin{equation}\label{eq:xi}
\xi\equiv\xi_\varGamma:=\sum_{x\in\calX} x\myp\nu(x)
\end{equation}
has the meaning of the \emph{right endpoint} of the corresponding
polygonal line $\varGamma$. In particular, the space $\CP_n$
($n\in\ZZ_+^2$) is identified as
$\CP_n=\{\varGamma\in\CP\colon\myp\xi_\varGamma=n\}$.

\subsubsection{Family\/ of multiplicative\/ measures\/ $\QQ_z$\myp.}\label{sec2.1.2} Let
$b_0,b_1,b_2,\dots$ be a sequence of nonnegative numbers such that
$b_0>0$ (without loss of generality, we put $b_0=1$) and \emph{not
all\/ $b_k$ vanish for $k\ge1$}, and assume that the corresponding
(ordinary) generating function
\begin{equation}\label{eq:beta}
\beta(u):=1+\sum_{k=1}^\infty b_k u^k,\qquad u\in\CC,
\end{equation}
is finite for $|u|<1$ (i.e., the radius of convergence of the power
series \eqref{eq:beta} is not smaller than~$1$). Let us now define a
family of probability measures $\QQ_{z}$ on the space
$\Phi=\ZZ_+^{\calX}$, indexed by the parameter $z=(z_1,z_2)\in
(0,1)\times(0,1)$, as the distribution of a random field
$\nu=\{\nu(x)\}_{x\in \calX}$ with mutually independent values and
marginal distributions
\begin{equation}\label{Q}
\QQ_{z}\{\nu(x)=k\}=\frac{b_k z^{k x}}{\beta(z^x)}\myp,\qquad
k=0,1,2,\dots\ \quad (x\in\calX).
\end{equation}

\begin{lemma}\label{pr:F0} For each $z\in (0,1)^2$, the
condition
\begin{equation}\label{N}
\tilde\beta(z):=\prod_{x\in \calX}\beta(z^{x})<\infty
\end{equation}
is necessary and sufficient in order that $\QQ_{z}(\Phi_0)=1$.
Furthermore, if\/ $\beta(u)$ is finite for all $|u|<1$ then the
condition \eqref{N} is satisfied.
\end{lemma}

\begin{proof} According to \eqref{Q},
$\QQ_{z}\{\nu(x)>0\}=1-\beta(z^x)^{-1}$ ($x\in\calX$). Hence,
Borel--Cantelli's lemma implies that $\QQ_{z}\{\nu\in\Phi_0\}=1$ if
and only if $\sum_{x\in\calX} \bigl(1-\beta(z^x)^{-1}\bigr)<\infty$.
In turn, the latter inequality is equivalent to~\eqref{N}.

To prove the second statement, observe using \eqref{eq:beta} that
\begin{equation}\label{eq:ln-beta0}
\ln\tilde{\beta}(z)=\sum_{x\in\calX}\ln \beta(z^x) \le
\sum_{x\in\calX} \bigl(\beta(z^x)-1\bigr)=\sum_{k=1}^\infty b_k
\sum_{x\in\calX} z^{kx}.
\end{equation}
Furthermore, for any $k\ge1$
\begin{align*}
\sum_{x\in \calX}z^{kx}&\le \sum_{x_1=1}^\infty z_1^{k x_1}+
\sum_{x_1=0}^\infty z_1^{kx_1}\sum_{x_2=1}^\infty z_2^{kx_2}\\
&=\frac{z_1^k}{1-z_1^k}+ \frac{z_2^k}{(1-z_1^k)(1-z_2^k)}\le
\frac{z_1^k}{1-z_1}+ \frac{z_2^k}{(1-z_1)(1-z_2)}\myp.
\end{align*}
Substituting this into \eqref{eq:ln-beta0} and recalling
\eqref{eq:beta}, we obtain
$$
\ln\tilde{\beta}(z) \le \frac{\beta(z_1)}{1-z_1}+
\frac{\beta(z_2)}{(1-z_1)(1-z_2)}<\infty,
$$
which implies~\eqref{N}.
\end{proof}

Lemma \ref{pr:F0} ensures that a sample configuration of the random
field $\nu(\cdot)$ belongs ($\QQ_{z}$-almost surely) to the space
$\Phi_0$, and therefore determines a \emph{finite} polygonal line
$\varGamma\in \CP$. By the mutual independence of the random values
$\nu(x)$ ($x\in\calX$), the corresponding $\QQ_z$-probability is
given by
\begin{equation}\label{Q1}
\QQ_{z}(\varGamma) =\prod_{x\in \calX}\frac{b_{\nu(x)}\myp
z^{x\myp\nu(x)}}{\beta(z^x)} =\frac{b(\varGamma)\myp
z^{\myp\xi}}{\tilde \beta(z)}\myp,\qquad \varGamma\in\CP,
\end{equation}
where $\xi:=\sum_{x\in\calX} x\myp\nu(x)$ (see the
definition~\eqref{eq:xi}) and
\begin{equation}\label{eq:b}
b(\varGamma):=\prod_{x\in \calX} \myn b_{\nu(x)}<\infty,\qquad
\varGamma\in\CP.
\end{equation}

\begin{remark}
The infinite product in \eqref{eq:b} contains only finitely many
terms different from $1$ (since $b_{\nu(x)}=b_0=1$ for
$x\notin\supp\myp\nu$).
%; hence, \eqref{eq:b} can be rewritten in an
%intrinsic form~\eqref{eq:b-Gamma}.
\end{remark}

In particular, for the trivial polygonal line
$\varGamma_\emptyset\leftrightarrow \nu\equiv0$ (see
Section~\ref{sec2.1.1}) the formula \eqref{Q1} yields
\begin{equation*}
%\label{eq:Gamma0}
\QQ_z(\varGamma_\emptyset)=\tilde \beta(z)^{-1}>0.
\end{equation*}
On the other hand, we have $\QQ_z(\varGamma_\emptyset)<1$, since
$\beta(u)>\beta(0)=1$ for all $u>0$ and hence, according to the
definition~\eqref{N}, $\tilde\beta(z)>1$ for any $z\in(0,1)^2$.

\subsubsection{Conditional measure $\PP_n$\myp.}
On the subspace $\CP_{n}\subset\CP$ of polygonal lines with the
right endpoint fixed at $n\in\ZZ_+^2$, the measure $\QQ_{z}$
($z\in(0,1)^2$) induces the conditional distribution
\begin{equation}\label{Pn}
\PP_n(\varGamma):=\QQ_{z}(\varGamma\myp|\myp\CP_n)
=\frac{\QQ_{z}(\varGamma)}{\QQ_{z}(\CP_n)}\myp,\qquad \varGamma\in
\CP_n.
\end{equation}
The formula \eqref{Pn} is well defined as long as
$\QQ_{z}(\CP_n)>0$, that is, there is at least one polygonal line
$\varGamma\in\CP_n$ with $b(\varGamma)>0$ (see
\eqref{Q1},~\eqref{eq:b}). A simple sufficient condition is as
follows.
\begin{lemma}
Suppose that $b_1>0$. Then $\QQ_z(\CP_n)>0$ for all $n\in\ZZ_+^2$
such that $n_1,n_2>0$.
\end{lemma}
\proof Observe that $n=(n_1,n_2)\in\ZZ_+^2$ (with $n_1,n_2\ge1$) can
be represented as
\begin{equation}\label{eq:n=n1+n2}
(n_1,n_2)=(n_1-1,1)+(1,n_2-1),
\end{equation}
where both points $x^{(1)}=(n_1-1,1)$ and $x^{(2)}=(1,n_2-1)$ belong
to the set $\calX$. Moreover, $x^{(1)}\ne x^{(2)}$ unless
$n_1=n_2=2$, in which case instead of \eqref{eq:n=n1+n2} we can
write $(2,2)=(1,0)+(1,2)$, where again $x^{(1)}=(1,0)\in\calX$,
$x^{(2)}=(1,2)\in\calX$. If $\varGamma^*\in\CP_n$ is a polygonal
line with two edges determined by the values $\nu(x^{(1)})=1$,
$\nu(x^{(2)})=1$ (and $\nu(x)=0$ otherwise), then, according to the
definition~\eqref{Q1}, $\QQ_z(\CP_n)\ge \QQ_z(\varGamma^*)=b_1^2
z^n\tilde{\beta}(z)^{-1}>0$.
\endproof

The parameter $z$ may be dropped in the notation \eqref{Pn} due to
the following key fact.

\begin{lemma}
%\label{lm:noz}
The measure\/ $\PP_n$ in\/ \eqref{Pn} does not depend on~$z$.
\end{lemma}

\begin{proof} If $\CP_n\ni\varGamma \leftrightarrow\nu_\varGamma\myn\in\Phi_0$ then
$\xi_\varGamma=n$ (see~\eqref{eq:xi}) and the formula \eqref{Q1} is
reduced to
\begin{equation*}
\QQ_{z}(\varGamma)= \frac{b(\varGamma)\myp z^n}{\tilde
\beta(z)}\myp,\qquad \varGamma\in \CP_n.
\end{equation*}
Accordingly, using \eqref{Pn} we get the expression
(cf.~\eqref{eq:P-r})
\begin{equation}\label{condP}
\PP_n(\varGamma)=\frac{b(\varGamma)}{\sum_{\varGamma'\myn\in\CP_n}\mynn
b(\varGamma')}\myp,\qquad \varGamma\in \CP_n,
\end{equation}
which is $z$-free.
\end{proof}

\subsection{Generating functions and cumulants}\label{sec2.2}

\subsubsection{Cumulant expansions.}
Recalling the expansion \eqref{eq:beta} for the generating function
$\beta(u)$ (with $\beta(0)=b_0=1$), consider the corresponding power
series expansion of its logarithm,
\begin{equation}\label{eq:ln-beta}
\ln\beta(u)=\sum_{k=1}^\infty a_k u^k,\qquad u\in\CC,
\end{equation}
assuming that the series \eqref{eq:ln-beta} is (absolutely)
convergent for all $|u| < 1$. Here and below, $\ln s$ with $s\in\CC$
means the principal branch of the logarithm specified by the value
$\ln 1=0$.

\begin{remark}\label{rm:b=a}
On substituting the expansion \eqref{eq:beta} into
\eqref{eq:ln-beta}, it is evident that $a_1=b_1$; more generally, if
$j^*\mynn:=\min\{j\ge 1\colon a_j\ne0\}$ and $k^*\mynn:=\min\{k\ge
1\colon b_k>0\}$ then $j^*\mynn=k^*$ and
$a_{j^*}\mynn=b_{k^*}\mynn>0$.
\end{remark}

Under the measure $\QQ_z$ defined in \eqref{Q}, the characteristic
function $\varphi_{\nu(x)}(t):=\EE_z(\rme^{\myp\rmi t\myp\nu(x)})$
of the random variable $\nu(x)$ ($x\in\calX$) is given by
\begin{equation}\label{eq:c.f.}
\varphi_{\nu(x)}(t)=\frac{\beta(z^x\rme^{\myp\rmi
t})}{\beta(z^x)}\myp, \qquad t\in\RR.
\end{equation}
[For notational simplicity, we suppress the dependence on $z$ in the
notation, which should cause no confusion.] Hence, with the help of
\eqref{eq:ln-beta} the (principal branch of the) logarithm of
$\varphi_{\nu(x)}(t)$ is expanded as
\begin{equation}\label{eq:ln_c.f.}
\ln \varphi_{\nu(x)}(t)=\ln \beta(z^x\rme^{\myp\rmi t})-\ln
\beta(z^x) =\sum_{k=1}^\infty a_k(\rme^{\myp\rmi kt}-1)\mypp
z^{kx},\qquad t\in\RR\myp.
\end{equation}

For a generic random variable $X$, let $\varkappa_q=\varkappa_q[X]$
denote its \emph{cumulants} of order $q\in\NN$ (see
\cite[\S\myp3.12, p.~69]{KS}), defined by the formal identity in
indeterminant $t$
\begin{equation}\label{eq:varkappa}
\ln \varphi(t)=\sum_{q=1}^\infty \frac{(\rmi\myp
t)^q}{q!}\mypp\varkappa_q,
\end{equation}
where $\varphi(t)=\EE\myp(\rme^{\rmi tX})$ is the characteristic
function of $X$. By differentiating \eqref{eq:varkappa} at $t=0$, it
is easy to see (cf.\ \cite[\S\myp3.14, Eq.~(3.37), p.~71]{KS}) that
\begin{equation}\label{eq:m1}
\EE\myp(X)=\varkappa_1,\qquad \Var\myp(X)=\varkappa_2\myp.
\end{equation}
Let us also point out the standard expressions for the next few
central moments of $X$ through the cumulants (see \cite[\S\myp3.14,
Eq.~(3.38), p.~72]{KS}): if $X_0:=X-\EE\myp(X)$ then
\begin{equation}\label{eq:kappa's}
\begin{aligned}
%\label{eq:m2}
\EE\myp(X_0^3)&=\varkappa_3,\\
\EE\myp(X_0^4)&=\varkappa_4 +
3\myp\varkappa_2^2,\\
%\label{eq:m5}
\EE\myp(X_0^5)&=\varkappa_5
+10\myp\varkappa_3\myp\varkappa_2,\\
%\label{eq:m6}
\EE\myp(X_0^6)&=\varkappa_6 +
15\myp\varkappa_4\myp\varkappa_2 + 10\myp\varkappa_3^2+
15\myp\varkappa_2^3.
\end{aligned}
\end{equation}

Let us now turn to the cumulants $\varkappa_q[\nu(x)]$ of the random
variables $\nu(x)$ (under the probability distribution $\QQ_z$). The
following simple lemma will be instrumental in our analysis.
\begin{lemma}\label{lm:kappa(q)}
The cumulants of\/ $\nu(x)$ \textup{(}$x\in\calX$\textup{)} are
given by
\begin{equation}\label{eq:cumulants}
\varkappa_q[\nu(x)]=\sum_{k=1}^\infty k^q a_k\myp z^{kx},\qquad
q\in\NN.
\end{equation}
\end{lemma}

\begin{proof}
Taylor expanding the exponential function in \eqref{eq:ln_c.f.}, we
get
\begin{equation}\label{eq:ln(g)}
\ln\varphi_{\nu(x)}(t)=\sum_{k=1}^\infty a_k z^{kx}
\sum_{q=1}^\infty \frac{(\rmi\myp kt)^q}{q!}
 =\sum_{q=1}^\infty
\frac{(\rmi\myp t)^q}{q!}\sum_{k=1}^\infty k^q a_k\myp z^{kx},
\end{equation}
where the interchange of the order of summation in the double series
\eqref{eq:ln(g)} is justified by its absolute convergence. Comparing
the expansion \eqref{eq:ln(g)} with the identity
\eqref{eq:varkappa}, we obtain the formulas \eqref{eq:cumulants} for
the coefficients $\varkappa_q[\nu(x)]$.
\end{proof}

Lemma \ref{lm:kappa(q)} allows us to obtain series representations
for the cumulants of the components
$\xi_j=\sum_{x\in\mathcal{X}}x_j\mypp\nu(x)$ of the random vector
$\xi=(\xi_1,\xi_2)$ (see~\eqref{eq:xi}). Namely, using the rescaling
relation $\varkappa_q[cX]=c^q\varkappa_q[X]$ (see \cite[\S\myp3.13,
p.~70]{KS}) and the additivity property of the cumulants for
independent summands (see \cite[\S\myp7.18, pp.\ 201--202]{KS}),
from \eqref{eq:cumulants} we get for $q\in\NN$
\begin{equation}\label{eq:cumulants-xi}
\varkappa_q[\myp\xi_j\myp]=\sum_{x\in\mathcal{X}}x_j^q
\myp\varkappa_q[\nu(x)]=\sum_{x\in\mathcal{X}}x_j^q
\sum_{k=1}^\infty k^q a_k\myp z^{kx}\qquad (j=1,2).
\end{equation}
In particular, the expected value and the variance of \myp$\xi_j$
are given by (see~\eqref{eq:m1})
\begin{align*}
%\label{eq:E-xi}
\EE_z(\xi_j)&=\sum_{x\in\mathcal{X}}x_j \sum_{k=1}^\infty k a_k\myp
z^{kx},\\
\Var_z(\xi_j)&=\sum_{x\in\mathcal{X}}x_j^2 \sum_{k=1}^\infty k^2
a_k\myp z^{kx}.
\end{align*}

\subsubsection{Dirichlet series associated with\/ $\ln\beta(u)$.}
For $s\in\CC$ such that $\Re\myp(s)=:\sigma>0$, consider the
Dirichlet series
\begin{equation}\label{eq:A2-}
A(s):=\sum_{k=1}^\infty \frac{a_{k}}{k^s}\myp,\qquad
A^+(\sigma):=\sum_{k=1}^\infty \frac{|a_{k}|}{k^\sigma}\myp,
\end{equation}
where $\{a_k\}$ are the coefficients in the power series expansion
of $\ln\beta(u)$ (see~\eqref{eq:ln-beta}).

Although some of the coefficients $\{a_k\}$ may be negative, it
turns out that the quantity $A(2) = \sum_{k=1}^\infty a_k\myp
k^{-2}$, whenever it is finite, cannot vanish or take a negative
value.

\begin{lemma}\label{lm:A(2)>0} If $A^+(2)<\infty$ then
$0<A(2)<\infty$ and the following integral formula holds
\begin{equation}\label{eq:kappa1}
A(2)=\int_0^1\mynn u^{-1}\mynn\left(\int_0^u\mynn v^{-1}
\ln\beta(v)\,\dif{v}\right)\dif{u}.
\end{equation}
\end{lemma}
\begin{proof}
From \eqref{eq:beta} it is evident that $\ln\beta(u)>\ln 1=0$ for
all $u\in(0,1)$, and it readily follows that the double integral on
right-hand side of \eqref{eq:kappa1} is positive (and possibly
infinite). Furthermore, substituting the expansion
\eqref{eq:ln-beta} and integrating term by term (which is
permissible for power series inside the interval of convergence), we
obtain for $s\in(0,1)$
\begin{align}
\notag \int_0^s\mynn u^{-1}\mynn\left(\int_0^u\myn v^{-1}
\ln\beta(v)\,\dif{v}\right)\dif{u}&=\int_0^s\mynn
u^{-1}\sum_{k=1}^\infty  a_k\myn \left(\int_0^u\mynn
v^{k-1}\,\dif{v}\right)\dif{u}\\
\label{eq:Abel} &=\sum_{k=1}^\infty  \frac{a_k}{k}\mynn
\int_0^s\mynn u^{k-1}\,\dif{u}=\sum_{k=1}^\infty
\frac{a_k}{k^2}\mypp s^k.
\end{align}
Passing here to the limit as $s\uparrow1$ and applying to the
right-hand side of \eqref{eq:Abel} Abel's theorem on power series
(see \cite[\S1.22, pp.\ 9--10]{Titch2}), we obtain the
identity~\eqref{eq:kappa1}.
\end{proof}

\begin{remark}
The condition $A^+(2)<\infty$ and the quantity $A(2)$ will play a
major role in our argumentation; in particular, $A(2)$ is involved
in a suitable calibration of the ``free'' parameter $z=(z_1,z_2)$ in
the definition \eqref{Q} of the measure $\QQ_z$ (see Section
\ref{sec3.1} below). However, some results (such as Theorem
\ref{th:4.1}, Lemma \ref{lm:7.1} and Theorem \ref{th:LCLT}) will
require a stronger condition $A^+(1)<\infty$. Our main result on the
limit shape under the measure $\PP_n$ (see Theorem~\ref{th:LSP}) is
dependent on these statements, and therefore is stated and proved
under the latter condition.
\end{remark}

\subsection{Auxiliary estimates for power-exponential
sums}\label{sec2.3}

In what follows, we frequently encounter power-exponential sums of
the form
\begin{equation}\label{eq:q}
S_q(t):=\sum_{k=1}^\infty k^{q-1}\rme^{-tk},\qquad t>0.
\end{equation}
For the first few integer values of $q$, explicit expressions of
$S_q(t)$ are easily available,
\begin{equation}\label{eq:Sq-exact}
S_1(t)=\frac{\rme^{-t}}{1-\rme^{-t}}\myp,\qquad
S_2(t)=\frac{\rme^{-t}}{(1-\rme^{-t})^2}\myp,\qquad
S_3(t)=\frac{\rme^{-t}\mypp(1+\rme^{-t})}{(1-\rme^{-t})^3}\myp.
\end{equation}

The purpose of this subsection is to obtain estimates on $S_q(t)$
with any integer~$q$.

\begin{lemma}\label{lm:6.2}
For\/ $q\in\NN$, the function $S_q(t)$ admits the representation
\begin{equation}\label{eq:S}
S_q(t)=\sum_{j=1}^{q} c_{j,\myp q}\,\frac{\rme^{-t
j}}{(1-\rme^{-t})^j},\qquad t>0,
\end{equation}
with some constants\/ $c_{j,\myp q}>0$
\,\textup{(}$j=1,\dots,q$\textup{)}\myp\textup{;} in particular,
$c_{1,\myp q}=1$ and $c_{q,\myp q}=(q-1)!$\mypp.
\end{lemma}

\proof For $q=1$, the expression for $S_1(t)$ from
\eqref{eq:Sq-exact} is a particular case of \eqref{eq:S} with
$c_{1,1}=1$. Assume now that the expansion \eqref{eq:S} is valid for
some $q\ge1$ (including the ``boundary'' values $c_{1,\myp q}=1$,
\,$c_{q,\myp q}=(q-1)!$\myp). Then, differentiating the identities
\eqref{eq:q} and \eqref{eq:S} with respect to $t$, we obtain
\begin{align*}
S_{q+1}(t)=-\frac{\dif}{\dif t} S_q(t)&=\sum_{j=1}^{q} c_{j,\myp
q}\left(\frac{j\mypp \rme^{-t j}}{(1-\rme^{-t})^j}+
\frac{j\mypp\rme^{-t
(j+1)}}{(1-\rme^{-t})^{j+1}}\right)\\
&=\sum_{j=1}^{q+1} c_{j,\myp q+1}\frac{\rme^{-t
j}}{(1-\rme^{-t})^j}\myp,
\end{align*}
where we set
\begin{equation*}
  c_{j,\myp{}q+1}:=
  \begin{cases}
  \ c_{1,\myp{}q}\myp,&j=1,\\
  \ j\myp c_{j,\myp{}q}+(j-1)\mypp c_{j-1,\myp{}q}\myp,\ \ &2\le j\le q,\\
  \ q\myp c_{q,\myp{}q}\myp,& j=q+1.
  \end{cases}
\end{equation*}
In particular, $c_{1,\myp q+1}= c_{1,\myp{}q}=1$ and $c_{q+1,\myp
q+1}=q\myp c_{q,\myp{}q}=q\myp(q-1)!=q!$\myp. Thus, the formula
\eqref{eq:S} holds for $q+1$ and hence, by induction, for all
$q\ge1$.
\endproof

\begin{lemma}\label{lm:C_k}
\textup{(a)} \,For each $q\in\NN$, there exists an absolute
constant\/ $\bar{c}_q>0$ such that
\begin{equation}\label{eq:Cc}
S_q(t)\le \frac{\bar{c}_q\mypp\rme^{-t}}{(1-\rme^{-t})^q}\myp,\qquad
t>0.
\end{equation}

 \textup{(b)} \,Moreover,
\begin{equation}\label{eq:Sq}
S_q(t)\sim \frac{(q-1)!}{t^q}\myp,\qquad t\to 0.
\end{equation}
\end{lemma}
\proof (a) \,Observe that for $j=1,\dots,q$ and all $t>0$
\begin{equation*}
\frac{\rme^{-tj}}{(1-\rme^{-t})^j}\le\frac{\rme^{-t}}{(1-\rme^{-t})^q}\myp.
\end{equation*}
Substituting these inequalities into \eqref{eq:S} and recalling that
all $c_{j,\myp q}>0$, we obtain \eqref{eq:Cc} with
$\bar{c}_q:=\sum_{j=1}^q c_{j,\myp q}$\mypp.

\smallskip
(b) \,For each term in the expansion \eqref{eq:S} we have $\rme^{-t
j}\myp(1-\rme^{-t})^{-j}\sim t^{-j}$ as $t\to0$. Hence, the overall
asymptotic behavior of $S_q(t)$ is determined by the term with $j=q$
and the corresponding coefficient $c_{q,\myp q}=(q-1)!$ (see
Lemma~\ref{lm:6.2}), and the formula \eqref{eq:Sq} follows.
\endproof

The next general lemma can be used to obtain a simplified polynomial
estimate for the right-hand side of the bound \eqref{eq:Cc}, which
is sometimes convenient.

\begin{lemma}\label{lm:exp<C}
For any\/ $q>0$ and $\theta>0$, there is a constant\/
$C_q(\theta)>0$ such that
\begin{equation}\label{eq:exp_bound}
\frac{\rme^{-\theta t}}{(1-\rme^{-t})^q} \le C_q(\theta)\mypp
t^{-q},\qquad t>0.
\end{equation}
\end{lemma}
\begin{proof}
Set $f(t):=t^q\mypp\rme^{-\theta t}(1-\rme^{-t})^{-q}$ and note that
\begin{equation*}
\lim_{t\downarrow 0}f(t)=1,\qquad \lim_{t\to +\infty} f(t)=0.
\end{equation*}
By continuity, the function $f(t)$ is bounded on $(0,\infty)$, and
the inequality \eqref{eq:exp_bound} follows.
\end{proof}

\section{Asymptotics of the expectation}
\label{sec3}

\subsection{Calibration of the parameter $z$}\label{sec3.1}

 Our aim in this section is to adjust the parameter
$z=(z_1,z_2)$, as a suitable function of $n=(n_1,n_2)$, in such a
way that under the corresponding measure $\QQ_z$ the following
asymptotic conditions are satisfied,
\begin{equation}\label{calibr1}
\lim_{n\to\infty} n_1^{-1} \EE_{z}(\xi_1)=\lim_{n\to\infty} n_2^{-1}
\EE_{z}(\xi_2)=1,
\end{equation}
where $\xi_j=\sum_{x\in\calX}x_j\mypp\nu(x)$ (see~\eqref{eq:xi}) and
$\EE_{z}$ denotes expectation with respect to $\QQ_{z}$. Let us use
the ansatz
\begin{equation}\label{alpha}
z_j=\rme^{-\alpha_j},\qquad \alpha_j=\delta_j\mypp n_j^{-1/3}\qquad
(j=1,2),
\end{equation}
where the quantities $\delta_1,\delta_2>0$, possibly depending on
$n$, are presumed to be bounded and separated from zero (i.e.,
$\delta_1, \delta_2\asymp 1$ as $n\to\infty$). Hence, using the
formula \eqref{eq:cumulants-xi} with $q=1$, we get (in vector form)
\begin{equation}\label{E_i'}
\EE_{z}(\xi)=\sum_{k=1}^\infty k a_{k} \sum_{x\in\calX}x\mypp
\rme^{-k\langle\alpha,\myp x\rangle}.
\end{equation}

\subsubsection{Evaluating sums over $\calX$ via the M\"obius inversion formula.}
Recall that the \textit{M\"obius function}
$\mu\colon\NN\to\{-1,0,1\}$ is defined as follows (see \cite[\S16.3,
p.~304]{HW}): $\mu(1):=1$, \,$\mu(m):=(-1)^{d}$ if $m$ is a product
of $d$ different prime numbers, and $\mu(m):=0$ otherwise; in
particular, $|\mu(m)|\le 1$ for all $m\in\NN$.

To deal with sums over the set $\calX$ (see~\eqref{eq:X}), the
following lemma will be instrumental.
\begin{lemma}\label{lm:Mobius}
Let $f\colon\RR_+^2\to\RR_+$ be a function such that
$f(x)|_{x=(0,0)}=0$ and for all $h>0$
\begin{equation}\label{eq:F}
F(h):= \sum_{x\in\ZZ^2_+}f(hx)<\infty.
\end{equation}
Moreover, assume that
\begin{equation}\label{eq:sum_sum<}
\sum_{k=1}^\infty F(hk)<\infty,\qquad h>0.
\end{equation}
Then the function
\begin{equation}\label{F_sharp}
F^\sharp(h):=\sum_{x\in \calX} f(hx), \qquad h>0,
\end{equation}
satisfies the identity
\begin{equation}\label{eq:Mobius}
F^\sharp(h)=\sum_{m=1}^\infty\mu(m)F(m h),\qquad h>0.
\end{equation}
\end{lemma}
\begin{proof}
Recalling the decomposition \eqref{eq:cone} and using that $f(x)$
vanishes at the origin, observe from \eqref{eq:F} and
\eqref{F_sharp} that
\begin{equation}\label{eq:F=sum}
F(h)=\sum_{m=1}^\infty \mynn F^\sharp(m h),\qquad h>0.
\end{equation}
Then the identity \eqref{eq:Mobius} follows by the M\"obius
inversion formula (see \cite[\S16.5, Theorem~270, p.~307]{HW}),
provided that $\sum_{k,\myp{}m} F^\sharp(km h)<\infty$ ($h>0$).
Indeed, the latter condition is satisfied,
\begin{align*}
\sum_{k=1}^\infty\sum_{m=1}^\infty F^\sharp(km h)&=
\sum_{k=1}^\infty F(kh)<\infty,
\end{align*}
according to \eqref{eq:F=sum} and the
hypothesis~\eqref{eq:sum_sum<}. This completes the proof.
\end{proof}

\subsubsection{The basic parameterization.}
\begin{theorem}\label{th:delta12}
Suppose that $A^+(2)<\infty$ \textup{(}see~\eqref{eq:A2-}\textup{)},
and choose $\delta_1,\delta_2$ in\/ \eqref{alpha} as follows
\begin{equation}\label{delta12}
\delta_1=\kappa\mypp \tau^{1/3},\qquad \delta_2=\kappa\mypp
\tau^{-1/3},
\end{equation}
where
\begin{equation}\label{eq:kappa}
\tau\equiv\tau_n:=\frac{n_2}{n_1}\myp,\qquad
\kappa\myn:=\left(\frac{A(2)}{\zeta(2)}\right)^{\myn1/3}.
\end{equation}
Then the asymptotic conditions\/ \eqref{calibr1} are satisfied.
\end{theorem}

\begin{remark}\label{rm:tau}
According to our convention about the limit $n\to\infty$ (see the
end of the Introduction), we have $\tau\asymp 1$. Observe also that
\eqref{alpha}, \eqref{delta12} and \eqref{eq:kappa} imply the
scaling relations
\begin{equation}\label{eq:alpha2:1}
\alpha^2_1\alpha_2\mypp n_1= \alpha_1\alpha^2_2\mypp
n_2=\kappa^3,\qquad \alpha_2=\alpha_1/\tau.
\end{equation}
\end{remark}

\begin{proof}[Proof of Theorem \textup{\ref{th:delta12}}]
Let us prove \eqref{calibr1} for $\xi_1$ (the proof for $\xi_2$ is
similar). Setting
\begin{equation}\label{f+}
f(x):=x_1\myp\rme^{-\langle\alpha,\myp
x\rangle}=x_1\myp\rme^{-\alpha_1 x_1-\alpha_2 x_2},\qquad
x=(x_1,x_2)\in\RR^2_+\myp,
\end{equation}
and following the notation \eqref{F_sharp} of Lemma \ref{lm:Mobius},
a projection of the equation \eqref{E_i'} onto the first coordinate
takes the form
\begin{equation}\label{E_1F}
\EE_{z}(\xi_1) =\sum_{k=1}^\infty
ka_{k}\sum_{x\in\calX}x_1\myp\rme^{-\langle\alpha,\myp
kx\rangle}=\sum_{k=1}^\infty a_{k}\sum_{x\in\calX} f(kx)=
\sum_{k=1}^\infty a_{k} F^\sharp(k).
\end{equation}
On the other hand, substituting \eqref{f+} into \eqref{F_sharp} and
using the expression \eqref{eq:Sq-exact} for $S_q(\cdot)$ with
$q=2$, we obtain
\begin{equation}\label{F}
F(h)=h\sum_{x_1=1}^\infty
x_1\myp\rme^{-h\alpha_1x_1}\sum_{x_2=0}^\infty \rme^{-h\alpha_2x_2}
=\frac{h\mypp
\rme^{-h\alpha_1}}{(1-\rme^{-h\alpha_1})^2(1-\rme^{-h\alpha_2})}\myp.
\end{equation}
It is evident that $F(h)$ satisfies the condition
\eqref{eq:sum_sum<}, hence by Lemma \ref{lm:Mobius}
%is satisfied, hence
the function $F^\sharp(\cdot)$ (see \eqref{F_sharp}) can be
expressed via the formula \eqref{eq:Mobius}. Thus, substituting also
\eqref{F}, we can rewrite \eqref{E_1F} as
\begin{equation}\label{E_1}
\EE_{z}(\xi_1)= \sum_{k=1}^\infty a_{k}\sum_{m=1}^\infty \mu(m) \myp
F(km) = \sum_{k,\myp m=1}^\infty \frac{k a_{k}\myp m\mu(m)\mypp
\rme^{-km\alpha_1}}
{(1-\rme^{-km\alpha_1\myn})^2\myp(1-\rme^{-km\alpha_2})}\myp.
\end{equation}

Now, using the representation \eqref{E_1} we can obtain the
asymptotics of $\EE_z(\xi_1)$ as $n\to\infty$. Recall that
$\alpha_1=\alpha_2\myp \tau$ (see~\eqref{eq:alpha2:1}), where
$\tau\equiv \tau_n\asymp 1$ (see \eqref{eq:kappa} and
Remark~\ref{rm:tau}) and so $\tau\ge\tau_*$ for some $\tau_*>0$ and
all $n$ large enough. Applying Lemma \ref{lm:exp<C} twice (with
$q=2$, \myp$\theta=1/2$ and $q=1$, \myp$\theta=\tau_*/2$,
respectively), we obtain, uniformly in $k,m\ge1$,
\begin{align}
\notag \frac{\alpha_1^{2}\alpha_2\mypp \rme^{-km\alpha_1}}
{(1-\rme^{-km\alpha_1\myn})^2\myp(1-\rme^{-km\alpha_2})}
&=\frac{\alpha_1^{2}\mypp \rme^{-km\alpha_1/2}}
{(1-\rme^{-km\alpha_1\myn})^2} \cdot\frac{\alpha_2\mypp
\rme^{-km\alpha_2\myp\tau/2}} {1-\rme^{-km\alpha_2}}\\
\label{eq:2,1} &\le \frac{C_2(1/2)}{(k\myp
m)^2}\cdot\frac{C_1(\tau_*/2)}{k\myp
m}=\frac{O(1)}{k^3m^3}\myp,\qquad n\to\infty.
\end{align}
Thus, remembering that $|\mu(m)|\le1$, the general summand in the
double sum \eqref{E_1}, multiplied by $\alpha_1^2\alpha_2$, is
bounded by $O(1)\myp |a_k|\mypp k^{-2}m^{-2}$, which is a term of a
convergent series due to the assumption $A^+(2)<\infty$. Hence, by
Lebesgue's dominated convergence theorem we get
\begin{align}
\notag \lim_{n\to\infty} \alpha_1^2\alpha_2\myp \EE_{z}(\xi_1)&=
\sum_{k,\myp m=1}^\infty  k a_{k}\myp m\mu(m)
\lim_{n\to\infty}\frac{\alpha_1^{2}\alpha_2\mypp
\rme^{-km\alpha_1}} {(1-\rme^{-km\alpha_1\myn})^2\myp(1-\rme^{-km\alpha_2})}\\
\label{zeta32*} &=\sum_{k=1}^\infty
\frac{a_{k}}{k^2}\sum_{m=1}^\infty
\frac{\mu(m)}{m^2}=\frac{A(2)}{\zeta(2)}\equiv\kappa^3,
\end{align}
according to the notation \eqref{eq:kappa}. Note that the identity
\begin{equation}\label{eq:zeta^{-1}}
\sum_{m=1}^\infty \frac{\mu(m)}{m^s}=\frac{1}{\zeta(s)}\myp,
\end{equation}
used in \eqref{zeta32*} for $s=2$, readily follows by the M\"obius
inversion formula \eqref{eq:Mobius} with $F^\sharp(h)=h^{-s}$ and
$F(h)=\sum_{m=1}^\infty (hm)^{-s} \allowbreak =h^{-s}\myp\zeta(s)$
\,(cf.\ \cite[\S17.5, Theorem~287, p.~326]{HW}).

To complete the proof, it remains to notice that the limit
\eqref{zeta32*} is equivalent to the first of the asymptotic
conditions \eqref{calibr1} due to the scaling relation
$\alpha_1^2\alpha_2=n_1^{-1}\myn\kappa^3$ (see~\eqref{eq:alpha2:1}).
\end{proof}

\begin{assumption}\label{as:z}
Throughout the rest of the paper, we assume that $A^+(2)<\infty$ and
the parameters $z_1,z_2$ are chosen according to the formulas
\eqref{alpha}, \eqref{delta12}, \eqref{eq:kappa}. In particular, the
measure $\QQ_z$ becomes dependent on $n=(n_1,n_2)$, as well as the
corresponding expected values.
\end{assumption}

\subsection{The ``expected'' limit shape}\label{sec3.2}
Given $n=(n_1,n_2)\in\ZZ^2_+$ \myp($n_1,\myp n_2>0$) and the ratio
$\tau=n_2/n_1$ (see~\eqref{eq:kappa}), for a polygonal line
$\varGamma\in\CP$ and $t\in[0,\infty]$ let us denote by
$\varGamma(t)\equiv \varGamma(t;\tau)$ \emph{the piece of
$\varGamma$ where the slope does not exceed $t\myp\tau$}. In case
all edges of $\varGamma$ have the slope bigger than $t\myp\tau$, we
set $\varGamma(t):=\varGamma_\emptyset$ (the trivial polygonal line,
see Section~\ref{sec2.1.1}).

\begin{remark}
The definition of $\varGamma(t)$ implies that under the scaling
$\mathfrak{s}_n$ (see \eqref{eq:scaling}) the scaled piece
$\tilde{\varGamma}_n(t):=\mathfrak{s}_n(\varGamma(t))$ has the slope
not bigger than~$t$.
\end{remark}

Consider the corresponding subset of $\calX$ (see \eqref{eq:X}),
\begin{equation}\label{eq:X_n}
\calX(t)\equiv \calX(t;\tau):=\{x=(x_1,x_2)\in
\calX\colon\,x_2/x_1\le t\myp\tau\}, \qquad t\in[0,\infty].
\end{equation}
According to the association
$\CP\ni\varGamma\leftrightarrow\nu\in\Phi_0$ described in
Section~\ref{sec2.1.1}, for each $t\in[0,\infty]$ the piece
$\varGamma(t)$ of $\varGamma$ is determined by a truncated
configuration $\{\nu(x),\,x\in\calX(t)\}$, hence its right endpoint
$\xi(t)=(\xi_1(t),\xi_2(t))$ is given by
\begin{equation}\label{xi(t)}
\xi(t)=\sum_{x\in \calX(t)}\myn x\myp \nu(x), \qquad t\in[0,\infty].
\end{equation}
In particular, $\calX(\infty)=\calX$, \,$\xi(\infty)=\xi$
(see~\eqref{eq:xi}). Similarly to \eqref{E_i'}, we have
\begin{equation}\label{eq:E(t)}
\EE_{z}[\myp\xi(t)]= \sum_{k=1}^{\infty} k a_{k}\sum_{x\in
\calX(t)}x\mypp \rme^{-k\langle\alpha,\myp x\rangle}, \qquad
t\in[0,\infty].
\end{equation}

Recall that the vector-function $g^*\myn(t)=(g_1^*(t),g_2^*(t))$ is
defined in~\eqref{eq:g*}.
\begin{theorem}\label{th:3.2}
Under Assumption \textup{\ref{as:z}}, for each\/ $t\in[0,\infty]$
\begin{equation}\label{sh}
\lim_{n\to\infty} n_j^{-1}\myn\EE_{z}[\myp\xi_j(t)]=g^*_j(t)\qquad
(j=1,2).
\end{equation}
\end{theorem}

\begin{proof}
Let $j=1$ (the case $j=2$ is considered in a similar manner).
Theorem \ref{th:delta12} implies that the claim \eqref{sh} holds for
$t=\infty$ (with $\xi_1(\infty)=\xi_1$). Thus, noting from
\eqref{eq:g*} that $g^*_1(\infty)=1$ and $1-g^*_1(t)=(1+t)^{-2}$, we
can rewrite \eqref{sh} (with $j=1$) in the form
\begin{equation}\label{sh1}
\lim_{n\to\infty}
n_1^{-1}\myn\EE_{z}[\myp\xi_1-\xi_1(t)]=(1+t)^{-2}.
\end{equation}

Now, like in the proof of Theorem \ref{th:delta12} (cf.\
\eqref{E_i'}, \eqref{E_1F} and~\eqref{E_1}), from \eqref{eq:E(t)} we
have
\begin{align}
\notag
\EE_{z}[\myp\xi_1-\xi_1(t)]&=\sum_{k=1}^{\infty} k
a_{k}\sum_{x\in
\calX\setminus \calX(t)}x_1\myp \rme^{-k\alpha_1x_1}\mypp \rme^{-k\alpha_2x_2}\\
& =\sum_{k,\myp m=1}^\infty k a_{k}\myp m\mu(m)\sum_{x_1=1}^\infty
x_1\myp\rme^{-km\alpha_1x_1}\sum_{x_2=\hat x_2+1}^{\infty}
\rme^{-km\alpha_2x_2}
\notag\\
\label{Z1} &=\sum_{k,\myp m=1}^\infty \frac{k a_{k}\myp
m\mu(m)}{1-\rme^{-km\alpha_2}}\sum_{x_1=1}^\infty
x_1\myp\rme^{-km\myp(\alpha_1x_1+\alpha_2(\hat{x}_2+1))},
\end{align}
where $\hat{x}_2=\hat{x}_2(t):= \lfloor t\myp \tau x_1\rfloor$, so
that
\begin{equation}\label{eq:x*}
0<\hat{x}_2+1-t\myp\tau x_1\le 1.
\end{equation}
It is natural to expect that the internal sum in \eqref{Z1} may be
well approximated by replacing $\hat{x}_2+1$ with $t\myp \tau x_1$
and thus reducing it to $S_2(km(\alpha_1+\alpha_2\myp t\myp\tau))$
(see the notation~\eqref{eq:q} with $q=2$). More precisely,
recalling that $\alpha_2\mypp\tau=\alpha_1$
(see~\eqref{eq:alpha2:1}), we obtain the representation
\begin{equation}\label{eq:S+R}
\sum_{x_1=1}^\infty
x_1\myp\rme^{-km\myp(\alpha_1x_1+\alpha_2(\hat{x}_2+1))}=
S_2(km\myp\alpha_1(1+t))-R_n(t;km),
\end{equation}
with
\begin{align}
%\label{eq:St} S_t(km\alpha):={}&\sum_{x_1=1}^\infty
\label{eq:Rt} R_n(t;km):={}& \sum_{x_1=1}^\infty
x_1\myp\rme^{-km\alpha_1x_1(1+t)}\!
\left(1-\rme^{-km\alpha_2\myp(\hat{x}_2 +1 -t\tau x_1)}\right).
\end{align}
By the expression \eqref{eq:Sq-exact} for $S_2(\cdot)$ we have
\begin{equation}\label{eq:S1}
0\le S_2(km\myp\alpha_1(1+t))
=\frac{\rme^{-km\alpha_1(1+t)}}{(1-\rme^{-km\alpha_1(1+t)})^2}\le
\frac{\rme^{-km\alpha_1}}{(1-\rme^{-km\alpha_1})^2}\myp.
\end{equation}
On the other hand, applying the upper inequality \eqref{eq:x*} under
the second exponent in \eqref{eq:Rt} and replacing $1+t$ by $1$
under the first exponent, we obtain the estimates
\begin{align}
\notag 0\le R_n(t;km) &\le (1-\rme^{-km\alpha_2})\sum_{x_1=1}^\infty
x_1\myp\rme^{-km\alpha_1x_1}\\
\label{eq:R1}
&=\frac{(1-\rme^{-km\alpha_2})\mypp\rme^{-km\alpha_1}}{(1-\rme^{-km\alpha_1\myn})^2}\\
\label{eq:R2}
&\le\frac{\rme^{-km\alpha_1}}{(1-\rme^{-km\alpha_1\myn})^2}\myp.
\end{align}

On substituting \eqref{eq:S+R} back into \eqref{Z1}, from the bounds
\eqref{eq:S1} and \eqref{eq:R2} it is evident that we can repeat the
arguments used in the proof of Theorem \ref{th:delta12}
(see~\eqref{eq:2,1}) and thus pass to the limit in \eqref{Z1} by
Lebesgue's dominated convergence theorem, giving
\begin{equation}\label{limE}
\lim_{n\to\infty} \alpha_1^2\alpha_2\myp
\EE_{z}[\myp\xi_1-\xi_1(t)]= \sum_{k,\myp m=1}^\infty  k a_{k}\myp
m\mu(m) \lim_{n\to\infty}\frac{\alpha_1^{2}\alpha_2\mypp
\bigl(S_2(km\myp\alpha_1(1+t))-R_n(t;km)\bigr)}
{1-\rme^{-km\alpha_2}}\myp.
\end{equation}
By virtue of the equality in \eqref{eq:S1} we easily find
\begin{equation}
\label{zeta320*} \lim_{n\to\infty}\frac{\alpha_1^{2}\alpha_2\mypp
S_2(km\myp\alpha_1(1+t)) } {1-\rme^{-km\alpha_2}}=\lim_{n\to\infty}
\frac{\alpha_1^{2}\alpha_2\mypp\rme^{-km\alpha_1(1+t)}}
{(1-\rme^{-km\alpha_2})\myp(1-\rme^{-km\alpha_1(1+t)})^2}=
\frac{1}{k^3\myp m^3\myp(1+t)^2}\myp.
\end{equation}
Furthermore, the estimate \eqref{eq:R1} implies
\begin{equation}\label{zeta320**}
\lim_{n\to\infty}\frac{\alpha_1^{2}\alpha_2\mypp R_n(t;km)}
{1-\rme^{-km\alpha_2}}\le
\lim_{n\to\infty}\frac{\alpha_1^{2}\alpha_2\mypp \rme^{-km\alpha_1}}
{(1-\rme^{-km\alpha_1\myn})^2} =0.
\end{equation}
Hence, substituting \eqref{zeta320*} and \eqref{zeta320**} into
\eqref{limE}, we obtain (cf.~\eqref{zeta32*})
\begin{equation}\label{eq:lim(t)}
\lim_{n\to\infty} \alpha_1^2\alpha_2\myp
\EE_{z}[\myp\xi_1-\xi_1(t)]= \sum_{k=1}^\infty
\frac{a_{k}}{k^2}\sum_{m=1}^\infty \frac{\mu(m)}{m^2}\,(1+t)^{-2}=
\kappa^3\myp(1+t)^{-2}.
\end{equation}
Finally, recalling that $\alpha_1^2\alpha_2=n_1^{-1}\myn\kappa^3$
(see~\eqref{eq:alpha2:1}), the limit \eqref{eq:lim(t)} is reduced
to~\eqref{sh1}.
\end{proof}

\subsubsection{Enhancement: uniform convergence.}
There is a stronger version of Theorem~\ref{th:3.2}.
\begin{theorem}\label{th:8.1.1a}
The convergence in\/ \eqref{sh} is uniform in\/ $t\in[0,\infty]$,
that is,
\begin{equation*}
\lim_{n\to\infty}\sup_{0\le t\le\infty} \bigl|\myp
n_j^{-1}\myn\EE_{z}[\myp\xi_j(t)]-g^*_j(t)\bigr|=0\qquad (j=1,2).
\end{equation*}
\end{theorem}

We use the following simple criterion of uniform convergence proved
in \cite[Lemma 4.3]{BZ4}.

\begin{lemma}\label{lm:8.1}
Let\/ $\{f_n(t)\}$ be a sequence of nondecreasing functions on a
finite interval\/ $[a,b\myp]$, such that, for each\/
$t\in[a,b\myp]$, $\lim_{n\to\infty}f_n(t)=f(t)$, where\/ $f(t)$ is a
continuous \textup{(}nondecreasing\textup{)} function on\/
$[a,b\myp]$. Then the convergence\/ $f_n(t)\to f(t)$ as\/
$n\to\infty$ is uniform on\/ $[a,b\myp]$.
\end{lemma}

\begin{proof}[Proof\/ of\/ Theorem\/ \textup{\ref{th:8.1.1a}}] Suppose that
$j=1$ (the case $j=2$ is handled similarly). Note that for each
$n=(n_1,n_2)$ (with $n_1>0$) the function
\begin{equation*}
f_n(t):=n_1^{-1}\myn\EE_{z}[\myp\xi_1(t)]=\frac{1}{n_1} \sum_{x\in
\calX(t)} x_1\myp \EE_{z}[\nu(x)],\qquad t\in[0,\infty],
\end{equation*}
is nondecreasing in $t$, in view of the definition \eqref{eq:X_n} of
the sets $\calX(t)$. Therefore, by Lemma \ref{lm:8.1} the
convergence in \eqref{sh} is uniform on any finite interval
$[0,t_0]$.

For large $t$, by the triangle inequality we get
\begin{equation}\label{eq:triangle}
|\myp n_1^{-1}\myn\EE_{z}[\myp\xi_1(t)]-g_1^*(t)|\le |\myp
n_1^{-1}\myn\EE_{z}(\xi_1)-1|+|\myp g_1^*(t)-1|+
n_1^{-1}\myn\EE_{z}[\myp\xi_1-\xi_1(t)]
\end{equation}
(in the last term, $\xi_1\ge \xi_1(t)$ for all $t\ge0$). We know
that $\lim_{n\to\infty} n_1^{-1}\myn\EE_{z}(\xi_1)=1$ by Theorem
\ref{th:3.2} and $\lim_{t\to\infty} g_1^*(t)=1$ (see~\eqref{eq:g*});
thus, in view of \eqref{eq:triangle} it remains to show that for any
$\varepsilon>0$ there is a $t_0=t_0(\varepsilon)$ such that, for all
large enough $n=(n_1,n_2)$ and all $t\ge t_0$,
\begin{equation}\label{eq:E<epsilon}
n_1^{-1}\myn\EE_{z}[\myp\xi_1-\xi_1(t)]\le \varepsilon.
\end{equation}
To this end, from the formulas \eqref{Z1} and  \eqref{eq:S+R} we
have
\begin{equation}\label{eq:Ez}
0\le \EE_{z}[\myp\xi_1-\xi_1(t)]\le\sum_{k,\myp m=1}^\infty \frac{k
|a_{k}|\myp m}{1-\rme^{-km\alpha_2}}\myp\bigl(
S_2(km\myp\alpha_1(1+t))+R_n(t;km)\bigr).
\end{equation}
For the part of the sum \eqref{eq:Ez} with
$S_2(km\myp\alpha_1(1+t))$, on substituting the equality
\eqref{eq:S1} and adapting the estimate \eqref{eq:2,1} derived in
the proof of Theorem \ref{th:delta12} we obtain for all $k,m\ge1$
and $t>0$
\begin{align*}
\frac{\alpha_1^2\alpha_2\myp
%\tilde{S}_n(t;km)
S_2(km\myp\alpha_1(1+t))}{1-\rme^{-km\alpha_2}}
&=\frac{\alpha_1^2\alpha_2\mypp \rme^{-km\alpha_1(1+t)}}
{(1-\rme^{-km\alpha_2})\myp(1-\rme^{-km\alpha_1(1+t)})^2} \le
\frac{C_1(\tau_*/2)\myp C_2(1/2)}{\mypp (k\myp
m)^{3}\myp(1+t)^{2}}\myp.
\end{align*}
Therefore, recalling that $\alpha_1^2\alpha_2=n_1^{-1}\myn\kappa^3$
(see~\eqref{eq:alpha2:1}) and using the condition $A^+(2)<\infty$,
we have uniformly in $t$ (and for all $n$)
\begin{equation}\label{eq:|S|}
\frac{1}{n_1}\sum_{k,\myp m=1}^\infty \frac{k |a_{k}|\myp
m}{1-\rme^{-km\alpha_2}}\mypp
%\tilde{S}_n(t;km)
S_2(km\myp\alpha_1(1+t))= \frac{O(1)}{(1+t)^2}\sum_{k,\myp
m=1}^\infty \frac{|a_{k}|}{k^2\myp m^2}=\frac{O(1)}{(1+t)^2}\le
\frac{\varepsilon}{2}\myp,
\end{equation}
provided $t$ is large enough.

On the other hand, by the dominated convergence argument
(cf.~\eqref{limE}) and due to the bound \eqref{eq:R1} leading to the
limit \eqref{zeta320**}, the contribution from $R_n(t;km)$ to the
sum \eqref{eq:Ez} is asymptotically negligible, uniformly in $t$,
which implies that for all $n$ large enough,
\begin{equation}\label{eq:|R|}
\frac{1}{n_1}\sum_{k,\myp m=1}^\infty \frac{k |a_{k}|\myp
m}{1-\rme^{-km\alpha_2}}\mypp R_n(t;km)\le
\frac{\varepsilon}{2}\myp.
\end{equation}
Thus, substituting the estimates \eqref{eq:|S|} and \eqref{eq:|R|}
into \eqref{eq:Ez}  yields \eqref{eq:E<epsilon} as desired, which
completes the proof of the theorem.
\end{proof}

\subsection{Refined asymptotics of the expectation}\label{sec3.3}

We need to sharpen the asymptotic estimate $\EE_z(\xi)-n=o(|n|)$
provided by Theorem~\ref{th:delta12}.

\begin{theorem}\label{th:4.1}
Under the condition $A^+(1)<\infty$, we have
$\EE_{z}(\xi)-n=O(|n|^{2/3})$ as $n\to\infty$.
\end{theorem}
For the proof of Theorem \ref{th:4.1}, some preparations are
required.

\subsubsection{Integral approximation of sums.}\label{sec3.3.1}
Let a function $f\colon \RR^2_+\to\RR_+$ be continuous and
integrable on $\RR^2_+$\myp, together with its partial derivatives
up to the second order. Set (cf.~\eqref{F_sharp})
\begin{equation}\label{F0}
F(h):=\sum_{x\in\ZZ^2_+}f(hx), \qquad h>0.
\end{equation}
Adapting the well-known Euler--Maclaurin formula (see, e.g.,
\cite[\S12.2]{Cramer}) to the double summation in \eqref{F0}, one
can verify (see more details in \cite[\S\myp5.1]{BZ4}) that the
above conditions on the function $f(x)$ ensure the absolute
convergence of the double series \eqref{F0} for any $h>0$ and,
moreover, $F(h)$ has the following asymptotics at the origin,
\begin{equation}\label{eq:EM}
\lim_{h\downarrow 0} \mypp h^2 F(h) = \int_{\RR^2_+}\mynn
f(x)\,\dif{x}<\infty.
\end{equation}
In particular, \eqref{eq:EM} implies that
\begin{equation}\label{beta_h->0}
F(h)=O(h^{-2}),\qquad h\downarrow 0.
\end{equation}
Furthermore, assume that for some $\beta>2$
\begin{equation}\label{beta_h->infty}
F(h)=O(h^{-\beta}),\qquad h\to+\infty,
\end{equation}
and consider the Mellin transform of $F(h)$ (see, e.g.,
\cite[Ch.\,VI, \S\myp9]{Widder}),
\begin{equation}\label{Mel}
\widehat{F}(s):=\int_0^\infty\mynn h^{s-1}F(h)\,\dif{}h\myp \qquad
(s\in\CC\myp).
\end{equation}
The estimates \eqref{beta_h->0}, \eqref{beta_h->infty} ensure that
the function $\widehat{F}(s)$ is well defined (and analytic) if
$2<\Re\myp(s)<\beta$. Moreover, $\widehat{F}(s)$ can be analytically
continued into the strip $1<\Re\myp(s)<2$. More precisely, consider
the function
\begin{equation}\label{Delta}
\varDelta_f(h):=F(h)-h^{-2}\!\int_{\RR^2_+}\myn f(x)\,\dif{x},
\qquad h>0,
\end{equation}
that is, the error in the approximation of the function $F(h)$ by
the corresponding integral (cf.~\eqref{eq:EM}). The following lemma
was proved in \cite[Lemma~5.2]{BZ4}.
\begin{lemma}\label{lm:Delta1}
Under the above conditions, the function\/ $\widehat{F}(s)$ defined
in \eqref{Mel} is meromorphic in the strip\/ $1<\Re\myp(s)<\beta$,
with a single\/ \textup{(}\myn{}simple\myp\textup{)} pole at\/
$s=2$. Moreover, $\widehat{F}(s)$ satisfies the identity
\begin{equation}\label{Muntz2}
\widehat{F}(s)=\int_0^\infty\mynn
h^{s-1}\varDelta_f(h)\,\dif{}h\myp,\qquad 1<\Re\myp(s)<2.
\end{equation}
\end{lemma}

\begin{remark}
The identity \eqref{Muntz2} is a two-dimensional analog of the
M\"untz formula for univariate functions (see \cite[\S\myp2.11, pp.\
28--29]{Titch1}).
\end{remark}

In turn, by the inversion formula for the Mellin transform (see
\cite[Theorem 9a, pp.\ 246--247]{Widder}), from \eqref{Muntz2} it
follows that, for any $c\in(1,2)$,
\begin{equation}\label{eq:inverse}
\varDelta_f(h)=\frac{1}{2\pi \rmi}\int_{c-\rmi\infty}^{c+\rmi\infty}
\myn h^{-s} \widehat{F}(s)\,\dif{}s
\end{equation}
(see \cite[Lemma 5.3]{BZ4} for more details).

\subsubsection{Proof\/ of\/ Theorem\/ \textup{\ref{th:4.1}}.}\label{sec3.3.2}

Our argumentation follows the same lines as in the proof of a
similar result in \cite[\S\myp5.2]{BZ4} for the special case of the
coefficients \eqref{eq:b-k-r} (with~$\rho=1$), but adapted to a more
general context based on the cumulant expansions. To be specific,
let us consider the coordinate $\xi_1$ of the random vector
$\xi=(\xi_1,\xi_2)$ (for $\xi_2$ the proof is similar).

\subsubsection*{Step\/ \textup{1}.} According to \eqref{E_1} we have
\begin{equation}\label{E_z''}
\EE_{z}(\xi_1)=\sum_{k,\myp m=1}^\infty a_{k}\myp \mu(m) F(km),
\end{equation}
where $F(h)$ is given by~\eqref{F}.
%(see~\eqref{F})
%\begin{equation}\label{F1}
%F(h)=\frac{h\mypp
%\rme^{-h\alpha_1}}{(1-\rme^{-h\alpha_1})^2(1-\rme^{-h\alpha_2})}\myp,\qquad
%h>0.
%\end{equation}
Note that the corresponding function
$f(x)=x_1\myp\rme^{-\langle\alpha,\myp x\rangle}$ (see~\eqref{f+})
has the property
\begin{equation}\label{eq:int(f)=}
\int_{\RR^2_+}\myn f(x)\,\dif{}x=\int_0^\infty\mynn
x_1\myp\rme^{-\alpha_1x_1}\,\dif{}x_1\int_0^\infty\mynn
\rme^{-\alpha_2\myp x_2}\,\dif{}x_2= \frac{1}{\alpha_1^2\myp
\alpha_2}\myp.
\end{equation}
Moreover, by virtue of the relation
$\alpha_1^2\alpha_2=\kappa^3/n_1$
%by virtue of \eqref{eq:kappa} and \eqref{eq:alpha2:1}
(see~\eqref{eq:alpha2:1}) we have (cf.~\eqref{zeta32*})
\begin{equation}\label{eq:int(f)}
\frac{1}{\alpha_1^2\alpha_2}\sum_{k,\myp m=1}^\infty \frac{a_{k}\myp
\mu(m)}{k^2\myp m^2}= \frac{n_1}{\kappa^3}\sum_{k=1}^\infty
\frac{a_{k}}{k^2}\sum_{m=1}^\infty \frac{\mu(m)}{m^2}\equiv n_1.
\end{equation}
Thus, subtracting \eqref{eq:int(f)} from \eqref{E_z''} and
substituting \eqref{eq:int(f)=} we obtain the representation
\begin{equation}\label{dif1}
\EE_{z}(\xi_1)-n_1=\sum_{k,\myp m=1}^\infty a_{k}\myp\mu(m)\mypp
\varDelta_{f}(km),
\end{equation}
where $\varDelta_f(h)$ is defined in~\eqref{Delta}.

\subsubsection*{Step\/ \textup{2}.}
Recalling the notation $\tau=n_2/n_1$ and the relation
$\alpha_2=\alpha_1/\tau$ (see \eqref{eq:kappa}
and~\eqref{eq:alpha2:1}), the Mellin transform \eqref{Mel} of the
function $F(h)$ may be represented in the form
%is reduced to
\begin{equation}\label{eq:F-tilde}
\widehat{F}(s)=\alpha_1^{-s-1} \tilde{F}(s),
\end{equation}
where
%the integral
\begin{equation}\label{eq:M_alpha}
\tilde{F}(s)=
%\alpha_1^{-s-1}\!
\int_0^\infty \frac{y^s\myp
\rme^{-y}}{(1-\rme^{-y})^2\mypp(1-\rme^{-y/\tau})}\,\dif{}y,\qquad
\Re\myp(s)>2\myp.
\end{equation}
Clearly, the functions $f(x)$, $F(h)$ satisfy all the hypotheses of
Section \ref{sec3.3.1}, including the asymptotics \eqref{beta_h->0}
and~\eqref{beta_h->infty}, with any $\beta>2$. Hence, by Lemma
\ref{lm:Delta1} the function $\widehat{F}(s)$ is regular for
$1<\Re\myp(s)<2$, and the formula \eqref{eq:inverse} together with
\eqref{eq:F-tilde} yields
\begin{equation}\label{eq:inverse*}
\varDelta_f(h)=\frac{1}{2\pi \rmi}\int_{c-\rmi\infty}^{c+\rmi\infty}
\myn h^{-s}\myp \alpha_1^{-s-1}\tilde{F}(s)\,\dif{}s,\qquad 1<c<2.
\end{equation}
Thus, substituting the representation \eqref{eq:inverse*} (with
$h=km$) into \eqref{dif1} we get
\begin{equation}\label{dif2}
\EE_{z}(\xi_1)-n_1=\frac{1}{2\pi \rmi}\sum_{k,\myp m=1}^\infty
a_{k}\myp\mu(m)\mypp \int_{c-\rmi\infty}^{c+\rmi\infty} \myn
\frac{\tilde{F}(s)}{\alpha_1^{s+1}(km)^s}\,\dif{}s, \qquad 1<c<2.
\end{equation}

\subsubsection*{Step\/ \textup{3}.} Aiming to mollify the
singularity of the integrand in \eqref{eq:M_alpha} at zero, set
\begin{equation}\label{eq:phi}
\phi(y):=\frac{y\mypp \rme^{-y}}{(1-\rme^{-y})^2}
\left(\frac{1}{1-\rme^{-y/\tau}}-\frac{\tau}{y}-\frac12\right)\mynn,\qquad
y>0,
\end{equation}
and consider the regularized integral
\begin{equation}\label{J}
\mathcal{I}(s):=\int_0^\infty\mynn y^{s-1}\myp\phi(y)\,\dif{y},
\end{equation}
so that \eqref{eq:M_alpha} is rewritten in the form
\begin{align}\label{tildeM1}
\tilde{F}(s)&=
%\alpha_1^{-s-1}\mynn\left(
\mathcal{I}(s)+\tau\mynn\int_0^\infty\mynn\frac{y^{s-1}\rme^{-y}}
{(1-\rme^{-y})^2}\,\dif{}y+\frac{1}{2}\int_0^\infty\mynn
\frac{y^s\myp \rme^{-y}}{(1-\rme^{-y})^2}\,\dif{}y.
%\right)\mynn.
\end{align}
The integrals in \eqref{tildeM1} are easily evaluated: if
$\Re\myp(s)>2$ then
\begin{align}
\notag
\int_0^\infty\mynn\frac{y^{s-1}\myp\rme^{-y}}{(1-\rme^{-y})^2}\,\dif{}y
&=\int_0^\infty\mynn y^{s-1}\sum_{k=1}^\infty k\mypp
\rme^{-ky}\,\dif{}y
=\sum_{k=1}^\infty k\int_0^\infty\mynn y^{s-1}\myp\rme^{-ky}\,\dif{}y\\
\label{in_1} &=\sum_{k=1}^\infty \frac{1}{k^{s-1}}
\myp\int_0^\infty\mynn u^{s-1}\mypp\rme^{-u}\,\dif{}u
=\zeta(s-1)\mypp\Gamma(s),
\end{align}
and likewise
\begin{equation}\label{in_2}
\int_0^\infty\mynn \frac{y^{s}\myp
\rme^{-y}}{(1-\rme^{-y})^2}\,\dif{}y= \zeta(s)\mypp\Gamma(s+1).
\end{equation}
Thus, substituting \eqref{in_1} and \eqref{in_2} into
\eqref{tildeM1} we get
\begin{equation}\label{tildeM2}
\tilde{F}(s)= \mathcal{I}(s)+
\tau\myp\zeta(s-1)\mypp\Gamma(s)+\tfrac{1}{2}\myp\zeta(s)\mypp\Gamma(s+1),
\qquad \Re\myp(s)>2.
\end{equation}

\subsubsection*{Step\/ \textup{4}.} The representation \eqref{tildeM2} renders an
explicit analytic continuation of $\widehat{F}(s)$ into the strip
$0<\Re\myp(s)<2$ (cf.\ Lemma~\ref{lm:Delta1}). To show this, let us
first investigate the integral~\eqref{J}.

\begin{lemma}\label{lm:phi}
The function $\phi(y)$ defined in \eqref{eq:phi} has the following
asymptotic expansions
\begin{alignat}{2}
\label{eq:phi@0}
\phi(y)&=\tfrac{1}{12}\mypp \tau^{-1}\mynn+O(y^2),\qquad &&y\to0,\\[.2pc]
\label{eq:phi@infty} \phi(y)&=\tfrac{1}{2}\myp y\mypp
\rme^{-y}\,(1+o(1)),\qquad &&y\to\infty,
\end{alignat}
which can be formally differentiated
%at least two times
to produce the corresponding expansions of\/ $\phi^{\myp\prime}(y)$,
$\phi^{\myp\prime\prime}(y)$.
\end{lemma}

\begin{proof}By Taylor's expansion it is easy to check that,
as $y\to0$,
\begin{align*}
\frac{1}{1-\rme^{-y/\tau}}-\frac{\tau}{y}-\frac12&=\frac{y}{12\tau}\left(1+O(y^2)\right)\mynn,\\
\frac{y\mypp
\rme^{-y}}{(1-\rme^{-y})^2}&=\frac{1}{y}\left(1+O(y^2)\right)\mynn,
\end{align*}
and \eqref{eq:phi@0} follows on substituting this
into~\eqref{eq:phi}. Since differentiation of Taylor expansions is
legitimate, from \eqref{eq:phi@0} we also get
$\phi^{\myp\prime}(y)=O(y)$ and $\phi^{\myp\prime\prime}(y)=O(1)$,
as $y\to0$.

The asymptotics \eqref{eq:phi@infty} follow immediately from
\eqref{eq:phi}, and it is also straightforward to see that the main
asymptotic contribution to the derivatives of $\phi(y)$, as
$y\to\infty$, is furnished by the term $y\mypp\rme^{-y}$, so that
$\phi^{\myp\prime}(y)\sim \frac12\myp y\mypp\rme^{-y}$ and
$\phi^{\myp\prime\prime}(y)\sim \frac12\myp y\mypp\rme^{-y}$ as
$y\to\infty$.
\end{proof}

In view of \eqref{eq:phi@0} and \eqref{eq:phi@infty}, the integral
\eqref{J} is absolutely convergent if $\Re\myp(s)>0$, and therefore
the function $\mathcal{I}(s)$ is regular in the corresponding
half-plane.

Returning to the representation \eqref{tildeM2}, note that the gamma
function $\Gamma(s)$ is analytic for $\Re\myp(s) > 0$ (see, e.g.,
\cite[\S\myp{}4.41, p.\;148]{Titch2}), whereas the Riemann zeta
function $\zeta(s)$ is meromorphic in the complex plane $\CC$ with a
single (simple) pole at point $s = 1$ (see, e.g.,
\cite[\S\myp{}4.43, p.\;152]{Titch2}). Thus, the right-hand side of
\eqref{tildeM2} is meromorphic in the half-plane $\Re\myp(s)>0$,
with the simple poles at $s=1$ and $s=2$\myp.

\subsubsection*{Step\/ \textup{5}.}
Setting $s=\sigma+\rmi t$, let us estimate the function
$\tilde{F}(s)$ as $t\to\infty$. First of all, integrating by parts
(twice) in \eqref{J} and using the asymptotic formulas
\eqref{eq:phi@0}, \eqref{eq:phi@infty} for the function $\phi(y)$
and its derivatives, we obtain
\begin{equation}\label{eq:I(s)<}
\mathcal{I}(s)=\frac{1}{s\myp(s+1)}\int_0^\infty\mynn
y^{s+1}\myp\phi^{\myp\prime\prime}(y)\,\dif{y}=O(t^{-2}),\qquad
t\to\infty,
\end{equation}
uniformly in $0<c_1\le \sigma\le c_2<\infty$.
%\begin{equation}\label{eq:I(s)<}
%\mathcal{I}(\sigma+\rmi
%t)=\frac{O(1)}{|s\myp(s+1)\myn|}=\frac{O(1)}{\sigma^2\myn+t^2}\myp.
%\end{equation}
The gamma function in such a strip satisfies the uniform estimate
(see \cite[\S\myp{}4.42, p.\;151]{Titch2})
\begin{equation}\label{GFE}
\Gamma(s)=O\bigl(|t|^{\sigma-1/2}\mypp \rme^{-\pi|t|/2}\bigr),\qquad
t\to\infty.
\end{equation}
%Furthermore, the Riemann zeta function is obviously bounded in any
%half-plane $\sigma\ge c_1>1$
%\begin{equation}\label{zeta_2}
%|\zeta(\sigma+it)|\le \sum_{n=1}^\infty
%\frac{1}{\left|n^{\sigma+it}\right|} =\sum_{n=1}^\infty
%\frac{1}{n^{\sigma}}\le
% \sum_{n=1}^\infty \frac{1}{n^{c_1}}=O(1).
%\end{equation}
We also have the following uniform bounds on the growth of the
Riemann zeta function as $t\to\infty$
%uniform in $\sigma$, on the growth of the zeta function as $t\to\infty$
(see \cite[Theorem 1.9, p.~25]{Iv}),
\begin{equation}\label{zeta_123}
\zeta(s)=\left\{
\begin{aligned}
{}\myp&O(\ln\myn|t|),&&\ \ 1\le\sigma\le2,\\
&O\bigl(|t|^{(1-\sigma)/2} \ln\myn |t|\bigr),&&\ \ 0\le\sigma\le1.
%\\
%&O\bigl(t^{1/2-\sigma} \ln\myn|t|\bigr),&&\ \ \sigma\le0.
\end{aligned}
\right.
\end{equation}
Therefore, substituting the estimates \eqref{eq:I(s)<}, \eqref{GFE}
and \eqref{zeta_123} into \eqref{tildeM2} and comparing the
resulting contributions, it is easy to check that for $1\le c\le 2$,
uniformly in $n\in\ZZ_+^2$\myp,
\begin{equation}\label{eq:F@infty}
\tilde{F}(c+\rmi t)=
%\alpha_1^{-c-1}
O(t^{-2}),\qquad t\to\infty.
\end{equation}

\subsubsection*{Step\/ \textup{6}.}
Interchanging the order of summation and integration in \eqref{dif2}
gives
\begin{align}
\notag \EE_{z}(\xi_1)-n_1&=\frac{1}{2\pi
\rmi}\int_{c-\rmi\infty}^{c+\rmi\infty}
\frac{\tilde{F}(s)}{\alpha_1^{s+1}} \sum_{k=1}^\infty
\frac{a_{k}}{k^s} \sum_{m=1}^\infty
\frac{\mu(m)}{m^s}\,\dif{}s\\
\label{dif3} &=\frac{1}{2\pi \rmi}\int_{c-\rmi\infty}^{c+\rmi\infty}
\frac{\tilde{F}(s)\myp A(s)}{\alpha_1^{s+1}\zeta(s)}\,\dif{}s,
\end{align}
where we used the notation \eqref{eq:A2-} and the formula
\eqref{eq:zeta^{-1}}. This computation is justified by virtue of the
absolute convergence, since for $1<c<2$
\begin{equation}\label{eq:abs_conv}
\int_{c-\rmi\infty}^{c+\rmi\infty}
\frac{|\tilde{F}(s)|}{\alpha_1^{c+1}}\sum_{k=1}^\infty
\frac{|a_{k}|}{k^c} \sum_{m=1}^\infty \frac{1}{m^c}\,\dif{}|s|=
%\alpha_1^{-c-1}
\frac{A^+(c)\mypp\zeta(c)}{\alpha_1^{c+1}}\int_{-\infty}^\infty
|\tilde{F}(c+\rmi t)|\,\dif{}t<\infty,
\end{equation}
where $A^+(c)\le A^+(1)<\infty$ due to the hypothesis of
Theorem~\ref{th:4.1}, whereas the last integral in
\eqref{eq:abs_conv} is finite thanks to the
bound~\eqref{eq:F@infty}.

Thus, on substituting \eqref{tildeM2} into \eqref{dif3} we have
\begin{equation}\label{dif4}
\EE_{z}(\xi_1)-n_1=\frac{1}{2\pi
\rmi}\int_{c-\rmi\infty}^{c+\rmi\infty}\mynn
A(s)\mypp\alpha_1^{-s-1}\mypp\varPsi(s)\,\dif{}s\qquad (1<c<2),
\end{equation}
where
\begin{equation}\label{Psi}
\varPsi(s):=\frac{\tilde{F}(s)}{\zeta(s)}=\frac{\mathcal{I}(s)+
\tau\myp\zeta(s-1)\myp\Gamma(s)}{\zeta(s)}+\frac{1}{2}\mypp\Gamma(s+1).
\end{equation}

\vspace{0.1pc}
\subsubsection*{Step\/ \textup{7}.} Since $\zeta(s)\ne 0$
for $\Re\myp(s) \ge 1$, the function $\varPsi(s)$ defined by the
expression \eqref{Psi} is analytic in the half-plane $\Re\myp(s)>1$;
moreover, it can be extended by continuity to the line
$\Re\myp(s)=1$, where the singularity at $s=1$ (due to the pole of
$\zeta(s)$ in the denominator) can be removed by setting
$\varPsi(1):=\lim_{s\to1}\varPsi(s)=\frac12\myp\Gamma(2)=\frac12$.

Let us show that the integration contour $\Re\myp(s)=c>1$ in
\eqref{dif4} can be moved to $\Re\myp(s)=1$. By the Cauchy theorem,
it suffices to check that
\begin{equation}\label{iT->0}
\lim_{T\to\pm\infty}\int_{1+\rmi\myp T}^{c+\rmi\myp T}\mynn
A(s)\mypp\alpha_1^{-s-1}\mypp\varPsi(s)\,\dif{}s=0.
%, \qquad \lim_{T\to\infty}\int_{1+\rmi
%T}^{c+\rmi T}\mynn A(s)\,\varPsi(s)\,\dif{}s=0.
\end{equation}
To this end, note that for $1\le \sigma\le c$
\begin{equation*}
|A(\sigma +\rmi\myp T)|\le A^+(\sigma)\le A^+(1)<\infty
\end{equation*}
(see \eqref{eq:A2-}) and
%the integral in \eqref{iT->0} is bounded by
%\begin{equation}\label{eq:F/zeta}
%%\int_{1}^{c}\mynn |A(\sigma +\rmi\myp T)|\cdot|\varPsi(\sigma
%%+\rmi\myp T)|\,\dif{}\sigma\le A^+(1)\mynn
%\int_{1}^{c}\mynn \frac{|A(\sigma +\rmi\myp
%T)|}{\alpha_1^{\sigma+1}}\cdot\frac{|\tilde{F}(\sigma +\rmi\myp
%T)|}{|\zeta(\sigma +\rmi\myp T)|}\:\dif{}\sigma,
%\end{equation}
%where   and
$\alpha_1^{-\sigma-1}\le \alpha_1^{-c-1}$ (since $\alpha_1\to0$, we
may assume that $\alpha_1<1$). From \eqref{eq:F@infty} we also have
$|\tilde{F}(\sigma+\rmi\myp T)|=O(T^{-2})$ as $T\to\infty$;
%
%for $s=\sigma+\rmi\myp t$ with
%$\sigma\ge1$ \sout{$1\le\sigma\le c<2$}
%\begin{equation}\label{|A|<}
%|A(\sigma+\rmi\myp t)|\le
%A^+(\sigma)\le A^+(1)<\infty,
%%\qquad \cancel{|\myp\alpha_1^{-\sigma-\rmi\myp
%%t-1}| \le \alpha_1^{-c-1}}\myp,
%\end{equation}
furthermore, it is known (see \cite[Eq.~(3.11.8), p.~60]{Titch1})
%, uniform in the strip $\Re\myp(s)\in[1,c\myp]$
that
%\begin{gather}
%\label{GFE} \Gamma(s)=O\myp(|t|^{\sigma-1/2}\mypp
%\rme^{-\pi|t|/2}),\qquad \Gamma(s+1)=O\myp(|t|^{\sigma+1/2}\mypp
%\rme^{-\pi|t|/2}),\\
%\label{zeta_123} \zeta(s-1)=O\myp(t^{1-\sigma/2}\ln\myn|t|),\qquad
%%\label{zeta_123-1}
%\zeta(s)=O\myp(\ln\myn|t|),\\[.3pc]
%\label{zeta_1}
%\zeta(\sigma+\rmi\myp t)^{-1}=O(\ln\myn|t|).
%\end{gather}
$\zeta(\sigma+\rmi\myp T)^{-1}=O(\ln\myn|T|)$ as $T\to\infty$,
uniformly in $\sigma\ge1$.
%Substituting the estimates \eqref{GFE} and \eqref{zeta_123} into
%\eqref{tildeM2}, we get $\widehat{F}(\sigma+\rmi\myp
%t)=\alpha_1^{-1-\sigma} O(t^{-2})$, and on account of \eqref{|A|<}
%and \eqref{J(it)} we see that \eqref{iT->0} follows.
Hence, substituting these estimates into \eqref{Psi} we obtain,
uniformly in $1\le \sigma\le c$ and $n\in\ZZ_+^2$\myp,
\begin{equation}\label{eq:Psi@infty}
\varPsi(\sigma+\rmi\myp T)=\frac{\tilde{F}(\sigma+\rmi\myp
T)}{\zeta(\sigma+\rmi\myp T)}=O\bigl(T^{-2}\ln\myn |T|\bigr)\to
0,\qquad T\to\infty.
\end{equation}
As a result, the limit \eqref{iT->0} follows. Thus, the
representation \eqref{dif4} takes the form
\begin{align}
\label{dif5} \EE_{z}(\xi_1)-n_1&=\frac{1}{2\pi
\rmi}\int_{1-\rmi\infty}^{1+\rmi\infty}\mynn
A(s)\mypp\alpha_1^{-s-1}\mypp\varPsi(s)\,\dif{}s.
\end{align}

\vspace{.1pc}
\subsubsection*{Step\/ \textup{8}.} Finally, the
formula \eqref{dif5} yields the bound
\begin{align}
\label{eq:|Psi|}
|\EE_{z}(\xi_1)-n_1|&=\frac{A^+(1)}{2\pi\alpha_1^{2}}\int_{-\infty}^\infty
|\varPsi(1+\rmi\myp t)|\,\dif{}t =O(|n|^{2/3}),
\end{align}
since $\alpha_1\asymp |n|^{-1/3}$ according to \eqref{alpha} and the
integral in \eqref{eq:|Psi|} is finite thanks to the bound
\eqref{eq:Psi@infty}. This completes the proof of
Theorem~\ref{th:4.1}.

%\begin{remark} If condition $A^+(\sigma)<\infty$ is satisfied
%with some $\sigma\in(0,1)$, then the statement of Theorem
%\ref{th:4.1} can be enhanced to $\EE_{z}(\xi)-n=o(|n|^{2/3})$ (cf.\
%\cite[p.~26??]{BZ4}). This is the case for all examples in
%Section~\ref{sec6} below.
%\end{remark}

\section{Asymptotics of higher-order moments}\label{sec4}
Throughout this section, we again assume that $A^+(2)<\infty$,
except in Section \ref{sec4.4} where a stronger condition $A^+(1)$
is required.

\subsection{The variance--covariance of \myp$\xi$}\label{sec4.1}

Denote $\mu_z:=\EE_{z}(\xi)$ and let $K_z:=\Cov_z(\xi,\xi)=
\EE_{z}(\xi-\mu_z)^{\myn\topp}(\xi-\mu_z)$ be the covariance matrix
of the random vector \,$\xi=\sum_{x\in\calX} x\myp\nu(x)$. Recalling
that the random variables $\nu(x)$ are independent for different
$x\in \calX$ and using \eqref{eq:cumulants} with $q=2$, the elements
$K_z(i,j)=\Cov_z(\xi_i,\xi_j)$ of the matrix $K_z$ are given by
\begin{equation}\label{D_zxi}
K_z(i,j)=\sum_{x\in \calX}x_i\myp x_j \myp\Var_z[\nu(x)]=\sum_{x\in
\calX} x_i\myp x_j \sum_{k=1}^\infty k^2 a_{k}\mypp z^{kx},\qquad
i,j\in\{1,2\}.
\end{equation}

\subsubsection{Asymptotics of\/ the covariance matrix.}\label{sec4.1.1}

\begin{theorem}\label{th:K}
As\/ $n\to\infty$,
\begin{equation}\label{eq:Sigma}
K_z(i,j)\sim  B_{ij}\,(n_1n_2)^{2/3},\qquad i,j\in\{1,2\},
\end{equation}
where the matrix $B=(B_{ij})$ is given by
\begin{equation}\label{eq:Sigma1}
\arraycolsep=.18pc B=\kappa^{-1}\mynn
\begin{pmatrix}2\myp \tau^{-1}&1\\
1&2\myp \tau\myn
\end{pmatrix}.
%\left(\arraycolsep=.20pc\begin{array}{cc}
%2\myp \tau^{-1}&1\\[.3pc]
%1&2\myp \tau\myn
%\end{array}
%  \right)\!.
\end{equation}
\end{theorem}

\begin{proof}
The calculations below follow the lines of the proof of
Theorem~\ref{th:delta12}, so we only sketch the proof. Let us first
consider the element $K_z(1,1)$. Substituting the parameterization
$z=\rme^{-\alpha}$ (see~\eqref{alpha}) into \eqref{D_zxi}, we obtain
(cf.~\eqref{E_i'})
\begin{equation}\label{D_zxi+}
K_z(1,1)=\sum_{x\in \calX} x_1^2\sum_{k=1}^\infty k^2
a_{k}\mypp\rme^{-k\langle\alpha,\myp x\rangle}.
\end{equation}
Using the M\"obius inversion formula \eqref{eq:Mobius}, similarly to
\eqref{E_1} the right-hand side of \eqref{D_zxi+} can be rewritten
in the form
\begin{align}
\notag K_z(1,1)&=\sum_{m=1}^\infty m^2\mu(m) \sum_{k=1}^\infty k^2
a_{k}\sum_{x\in\ZZ^2_+}
x_1^2\myp\rme^{-km\langle\alpha,\myp x\rangle}\\
\notag &=\sum_{k,\myp m=1}^\infty m^2\mu(m)\mypp k^2 a_{k}
\sum_{x_1=1}^\infty x_1^2\myp\rme^{-km\alpha_1 x_1}
\sum_{x_2=0}^{\infty}
\rme^{-km\alpha_2 x_2}\\
\label{D1_1(t)} &=\sum_{k,\myp m=1}^\infty m^2\mu(m)\mypp k^2
a_{k}\, \frac{\,\rme^{-km\alpha_1} (1+\rme^{-km\alpha_1\myn})}
{(1-\rme^{-km\alpha_1\myn})^3(1-\rme^{-km\alpha_2})}\myp,
\end{align}
where we used the expressions \eqref{eq:Sq-exact} for $S_1(\cdot)$
and $S_3(\cdot)$. Similarly to the estimate \eqref{eq:2,1}, by
virtue of  \eqref{eq:exp_bound} the general term in the sum
\eqref{D1_1(t)} is bounded by $O(\alpha_1^{-3}\alpha_2^{-1})\myp
|a_{k}|\myp k^{-2}\myp m^{-2}$, uniformly in $k,m$, and furthermore,
\begin{equation*}
\sum_{k,\myp m=1}^\infty \frac{|a_{k}|}{k^2\myp m^2}=
A^+(2)\mypp\zeta(2) <\infty.
\end{equation*}
Therefore, by the dominated convergence argument, from
\eqref{D1_1(t)} we obtain, similarly to~\eqref{eq:2,1},
\begin{equation}\label{eq:->A4}
\lim_{n\to\infty}\alpha_1^3\alpha_2\myp K_z(1,1) =2\sum_{k=1}^\infty
\frac{a_{k}}{k^2}\sum_{m=1}^\infty \frac{\mu(m)}
{m^2}=\frac{2A(2)}{\zeta(2)}=2\kappa^3.
\end{equation}
To reduce this limit to \eqref{eq:Sigma}, observe using the scaling
relations \eqref{eq:alpha2:1} that
\begin{equation}\label{eq:alpha*alpha}
\alpha_1^{3}\alpha_2=\frac{\alpha_1}{\alpha_2}\mypp\alpha_1^{2}\alpha_2^{2}=\tau\kappa^{4}(n_1n_2)^{-2/3},
\end{equation}
and from \eqref{eq:->A4} we get
\begin{equation}\label{eq:->B11}
\lim_{n\to\infty}(n_1n_2)^{-2/3}\myp
K_z(1,1)=\tau^{-1}\kappa^{-4}\lim_{n\to\infty}\alpha_1^3\myp\alpha_2\myp
K_z(1,1)=2\myp\tau^{-1}\kappa^{-1}=B_{11},
\end{equation}
as required (cf.\ \eqref{eq:Sigma},~\eqref{eq:Sigma1}).

The element $K_z(2,2)$ is analyzed in a similar fashion. Finally,
for $K_z(1,2)$ we obtain, similarly as in \eqref{D1_1(t)}
and~\eqref{eq:->A4},
\begin{align*}
K_z(1,2)&=\sum_{x\in \calX} x_1\myp x_2\sum_{k=1}^\infty k^2
a_{k}\mypp\rme^{-k\langle\alpha,\myp x\rangle}\\
&=\sum_{k,\myp m=1}^\infty k^2 a_{k}\, m^2\myn\mu(m)
\sum_{x_1=1}^\infty x_1\myp\rme^{-km\alpha_1 x_1}
\mynn\sum_{x_2=1}^{\infty}
x_2\mypp \rme^{-km\alpha_2 x_2}\\
\label{D1_1(t)} &=\sum_{k,\myp m=1}^\infty k^2 a_{k}\,
m^2\myn\mu(m)\, \frac{\rme^{-km\alpha_1}
\mypp\rme^{-km\alpha_2\myn}}{(1-\rme^{-km\alpha_1\myn})^2
\myp(1-\rme^{-km\alpha_2})^2}\\
&\sim \alpha_1^{-2}\alpha_2^{-2}\sum_{k=1}^\infty\frac{a_k}{k^2}
\sum_{m=1}^\infty\frac{\mu(m)}{m^2}= \alpha_1^{-2}\alpha_2^{-2}
\kappa^3\qquad (n\to\infty).
\end{align*}
Hence, using the identity
$\alpha_1^{2}\alpha_2^2=\kappa^{4}(n_1n_2)^{-2/3}$
(cf.~\eqref{eq:alpha*alpha}), it follows as in \eqref{eq:->B11} that
\begin{equation*}
\lim_{n\to\infty}(n_1n_2)^{-2/3}\myp
K_z(1,2)=\kappa^{-4}\lim_{n\to\infty}\alpha_1^2\myp\alpha_2^2\myp
K_z(1,2)=\kappa^{-1}\myn=B_{12},
\end{equation*}
according to the notation \eqref{eq:Sigma1}. Thus, the proof of the
theorem is complete.
\end{proof}

\subsubsection{The norm of\/ the covariance matrix.}\label{sec4.1.2}

The next lemma is an immediate corollary of Theorem~\ref{th:K}.
\begin{lemma}\label{lm:detK} As\/
$n\to\infty$,
\begin{equation}
\label{eq:detK} \det K_z\sim 3\myp \kappa^{-2}(n_1n_2)^{4/3}\asymp
|n|^{8/3}.
\end{equation}
\end{lemma}

This result
%Lemma \ref{lm:detK}
implies that the matrix $K_z$ is
non-degenerate, at least asymptotically as $n\to\infty$. In fact,
from \eqref{D_zxi} it is easy to deduce (e.g., using the
Cauchy--Schwarz inequality together with the characterization of the
equality case) that $K_z$ is positive definite; in particular, $\det
K_z>0$ and hence $K_z$ is invertible. Let $V_z:=K_z^{-1/2}$ be the
(unique) square root of the matrix $K_z^{-1}$, that is, a symmetric,
positive definite matrix such that $V_z^2=K_z^{-1}$.

We need some general facts about the matrix norm $\|{\cdot}\|$,
which we state as a lemma (see \cite[\S\myp7.2, p.~2301]{BZ4} for
simple proofs and bibliographic comments).

\begin{lemma}\label{lm:7.1.1}
\textup{(a)} \,If\/ $A$ is a real matrix then $\|A^{\topp}\mynn
A\|=\|A\|^2$.

%\begin{lemma}
%[see \mbox{\cite[Lemma~7.4]{BZ4}}]
%[\mbox{\normalfont cf.\ \cite[\S\myp{}5.6, Problem\,23, hints (2,5) and (5,2),
%pp.\ 313--314]{Horn}}]
%\label{lm:|A|}

\smallskip
\textup{(b)} \,If\/ $A=(a_{ij})$ is a\/ real $d\times d$
matrix\myp, then\/
\begin{equation*}
%\label{eq:m-norm}
\frac{1}{d}\sum_{i,\myp{}j=1}^d a_{ij}^2\le
\|A\|^2\le\sum_{i,\myp{}j=1}^d a_{ij}^2\myp.
\end{equation*}

%\label{lm:7.1.2}
\textup{(c)} \,Let\/ $A$ be a real symmetric\/ $2\times 2$ matrix
with\/ $\det A\ne 0$. Then
\begin{equation*}
%\label{eq:d=2}
\|A^{-1}\|=\frac{\|A\|}{|\mynn\det A|}\myp.
\end{equation*}
\end{lemma}

Let us now estimate the norm of the matrices $K_z$ and
$V_z=K_z^{-1/2}$.
\begin{lemma}\label{lm:K_z}
As\/ $n\to\infty$, one has\/
\begin{equation}\label{eq:KV}
\|K_z\|\asymp |n|^{4/3},\qquad \|V_z\|\asymp |n|^{-2/3}.
\end{equation}
\end{lemma}

\begin{proof}
Lemma \ref{lm:7.1.1}\myp(b) and Theorem \ref{th:K} imply
\begin{equation*}
\|K_z\|^2\asymp \sum_{i,\myp{}j=1}^2 K_z(i,j)^2\asymp
(n_1n_2)^{4/3}\asymp |n|^{8/3}\qquad (n\to\infty),
\end{equation*}
which proves the first estimate in~\eqref{eq:KV}. Furthermore, using
parts (a) and (c) of Lemma \ref{lm:7.1.1}, we have
\begin{equation*}
\|V_z\|^2=\|V_z^2\|=\|K_z^{-1}\|=\frac{\|K_z\|}{\det
K_z}\asymp\frac{|n|^{4/3}}{|n|^{8/3}}=|n|^{-4/3},
\end{equation*}
according to the known asymptotics of $\det K_z$ and $\|K_z\|$ (see
\eqref{eq:detK} and \eqref{eq:KV}, respectively). Hence, the second
estimate in \eqref{eq:KV} follows, and the proof of the lemma is
complete.
\end{proof}

\subsection{The cumulants of \myp$\xi_j$}\label{sec4.3}

By the parameterization $z=\rme^{-\alpha}$ (see~\eqref{alpha}), the
expansion \eqref{eq:cumulants-xi} takes the form
\begin{equation}\label{varkappa_zxi+}
\varkappa_q[\myp\xi_j] =\sum_{x\in \calX} x_j^q\sum_{k=1}^\infty k^q
a_{k}\mypp \rme^{-k\langle \alpha,\myp x\rangle},\qquad q\in\NN.
\end{equation}

\begin{lemma}\label{lm:kappa_q}
For each $q\in\NN$, as $n\to\infty$,
\begin{equation}\label{eq:kappa_q}
\varkappa_q[\myp\xi_j]\sim
\frac{q!\,\kappa^3}{\alpha_1^{q+1}\alpha_2} \asymp
|n|^{(q+2)/3},\qquad n\to\infty.
\end{equation}
\end{lemma}
\proof Let $j=1$ (the case $j=2$ is handled in a similar fashion).
Using the M\"obius inversion formula \eqref{eq:Mobius}, similarly to
\eqref{E_1} the right-hand side of \eqref{varkappa_zxi+} (with
$j=1$) can be rewritten as
\begin{align}
\notag \varkappa_q[\myp\xi_1]&=\sum_{k,\myp m=1}^\infty k^q
a_{k}\mypp m^q\mu(m)\sum_{x\in\ZZ^2_+}
x_1^q\myp\rme^{-km\langle\alpha,\myp x\rangle}\\
\notag &=\sum_{k,\myp m=1}^\infty k^q a_{k}\mypp m^q\mu(m)
\sum_{x_1=1}^\infty x_1^q\myp\rme^{-km\alpha_1 x_1}
\sum_{x_2=0}^{\infty}
\rme^{-km\alpha_2 x_2}\\
\label{D1_1(t)'} &=\sum_{k,\myp m=1}^\infty k^q a_{k}\mypp
m^q\mu(m)\mypp
S_{q+1}(km\alpha_1)\,\frac{1}{1-\rme^{-km\alpha_2}}\myp,
\end{align}
where we used the notation~\eqref{eq:q}. Applying Lemma
\ref{lm:C_k}\myp(a) and then the estimate \eqref{eq:exp_bound}
(cf.~\eqref{eq:2,1}), we obtain
\begin{align*}
\frac{\alpha_1^{q+1}\myn\alpha_2\mypp
S_{q+1}(km\alpha_1)}{1-\rme^{-km\alpha_2}}&\le
\frac{\alpha_1^{q+1}\myn\alpha_2\mypp
\bar{c}_{q+1}\mypp\rme^{-km\alpha_1}}{(1-\rme^{-km\alpha_1\myn})^{q+1}\myp(1-\rme^{-km\alpha_2\myn})}\\
&\le \frac{\bar{c}_{q+1}\mypp C_{q+1}(1/2)\mypp
C_1(\tau_*/2)}{(k\myp m)^{q+2}}\myp.
\end{align*}

Consequently, the general summand in the double series
\eqref{D1_1(t)'} multiplied by $\alpha_1^{q+1}\myn\alpha_2$ is
bounded, uniformly in $k,m\ge1$, by $O(1)\mypp |a_k|\myp
k^{-2}m^{-2}$, which is a term of a convergent series owing to the
condition $A^+(2)<\infty$. Hence, applying Lebesgue's dominated
convergence theorem and using Lemma \ref{lm:C_k}\myp(b), we obtain
%, as $\alpha_1,\alpha_2\to0$,
\begin{align}\label{eq:q!kappa}
\lim_{n\to\infty}\alpha_1^{q+1}\myn\alpha_2\mypp\varkappa_q[\myp\xi_1]
&= q!\sum_{k,\myp m=1}^\infty \frac{\mu(m)\mypp a_{k}}{k^2
m^2}=q!\,\frac{A(2)}{\zeta(2)}=q!\,\kappa^3\myp.
\end{align}
Finally, according to \eqref{alpha} we have
$\alpha_1^{q+1}\alpha_2\asymp |n|^{-(q+2)/3}$, and hence
\eqref{eq:q!kappa} implies~\eqref{eq:kappa_q}.
\endproof

In Section \ref{sec5.4} we will require an asymptotic bound for the
\emph{sixth-order central moment} of $\xi_j$, which is established
next.

\begin{lemma}\label{lm:6.7}
Set\/ \myp$\xi^0_j:=\xi_j-\EE_{z}(\xi_j)$
\mypp\textup{(}$j=1,2$\textup{)}. Then
\begin{equation}\label{eq:E_xi_6}
\EE_{z}\myn\bigl[(\xi^0_j)^{6}\bigr] \asymp |n|^{4},\qquad
n\to\infty.
\end{equation}
\end{lemma}
\begin{proof}
Using the expression of the sixth central moment via the cumulants
(see \eqref{eq:kappa's} with $X=\xi_j$ and $q=6$), we have
\begin{equation}\label{eq:kappa_6}
\EE_{z}\myn\bigl[(\xi^0_j)^{6}\bigr]=\varkappa_6[\myp\xi_j\myp]+
15\myp\varkappa_4[\myp\xi_j\myp]\myp\varkappa_2[\myp\xi_j\myp] +
10\myp(\varkappa_3[\myp\xi_j\myp])^2+
15\myp(\varkappa_2[\myp\xi_j\myp])^3.
\end{equation}
Applying Lemma \ref{lm:kappa_q} to the cumulants involved in
\eqref{eq:kappa_6}, it is easy to check that the main asymptotic
term is given by $(\varkappa_2[\myp\xi_j\myp])^3\asymp |n|^4$, which
proves the relation~\eqref{eq:E_xi_6}.
\end{proof}

%\begin{remark}
%Using the formulas \eqref{eq:kappa's} and their analogs for higher
%orders, one can prove the general asymptotic formula
%$\EE_{z}\myn\bigl[(\xi^0_j)^{q}\bigr]\asymp
%(\varkappa_2[\myp\xi_j\myp])^3\asymp |n|^4
%\end{remark}

\subsection{The Lyapunov coefficient}\label{sec4.4}
Let us introduce the \emph{Lyapunov coefficient} (of the third
order)
\begin{equation}\label{L3}
L_z:=\|V_z\|^{3} \sum_{x\in \calX}|x|^3\mu_3[\nu(x)],
\end{equation}
where $\mu_3[\nu(x)]$ is the third-order absolute central moment of
$\nu(x)$,
\begin{equation}\label{eq:nu0}
\mu_3[\nu(x)] :=\EE_{z}\myn\bigl[|\nu^0(x)|^3\bigr],\qquad
\nu^0(x):=\nu(x)-\EE_z[\nu(x)].
\end{equation}

The next asymptotic estimate will play an important role in the
proof of the local limit theorem in Section \ref{sec5.3} below.

\begin{lemma}\label{lm:7.1}
Suppose that $A^+(1)<\infty$. Then
\begin{equation}\label{eq:L}
L_z\asymp |n|^{-1/3},\qquad n\to\infty.
\end{equation}
\end{lemma}

\begin{proof}
In view of the definition \eqref{L3} and the asymptotics
$\|V_z\|\asymp |n|^{-2/3}$ (see~\eqref{eq:KV}), for the proof of
\eqref{eq:L} it suffices to show that
\begin{equation}\label{eq:sum(mu3)}
M_3:=\sum_{x\in\calX} |x|^{3} \mu_3[\nu(x)]\asymp |n|^{5/3},\qquad
n\to\infty.
\end{equation}
Starting with a \emph{lower bound} for $M_3$, observe using the
relation \eqref{eq:kappa's} with $q=3$ that
\begin{equation}\label{eq:mu3lower}
\mu_3[\nu(x)] \ge \EE_z[\nu^0(x)^3]=\varkappa_3[\nu(x)].
\end{equation}
Hence, using the formula \eqref{eq:cumulants-xi} and
Lemma~\ref{lm:kappa_q} (with $q=3$), from \eqref{eq:sum(mu3)} we get
\begin{equation}\label{eq:M-lower}
M_3\ge \sum_{x\in\calX} |x|^{3} \varkappa_3[\nu(x)]\ge
\sum_{x\in\calX} x_1^{3}
\myp\varkappa_3[\nu(x)]=\varkappa_3[\myp\xi_1]\asymp
|n|^{5/3},\qquad n\to\infty,
\end{equation}
which is in agreement with the claim~\eqref{eq:sum(mu3)}.

Let us now obtain a suitable \emph{upper bound} on $M_3$. First,
using the elementary inequality
$$|x|^3=(x_1^2+x_2^2)^{3/2}\le
\sqrt{2}\,(x_1^3+x_2^3)
$$
(which follows from H\"older's inequality for the function
$y=x^{3/2}$), we have
\begin{equation}\label{eq:M3<}
M_3\le \sqrt{2}\myp\sum_{x\in\calX} (x_1^{3}+x_2^3)\myp
\mu_3[\nu(x)].
\end{equation}
To estimate the moment $\mu_3[\nu(x)]$ (see \eqref{eq:nu0}), observe
that for any $u,v\ge0$
\begin{equation}\label{eq:|u-v|}
|u-v|^3=(u-v)^2|u-v|\le (u-v)^2(u+v)=(u-v)^3 +2\myp v\myp (u-v)^2.
\end{equation}
Setting in \eqref{eq:|u-v|} $u=\nu(x)$, \,$v=\EE_z[\nu(x)]$ and
taking the expectation, we get the inequality
\begin{align*}
%\notag
\mu_3[\nu(x)]&\le \EE_z[\nu^0(x)^3] +2\myp \EE_z[\nu(x)]\cdot
\EE_z[\nu^0(x)^2]\\[.1pc]
%\label{eq:mu3upper}
&=\varkappa_3[\nu(x)]+2\myp\varkappa_1[\nu(x)]\myp\varkappa_2[\nu(x)],
\end{align*}
according to the identities \eqref{eq:m1}, \eqref{eq:kappa's}
applied to $\nu(x)$. Note that the term $\varkappa_3[\nu(x)]$ here
is the same as in \eqref{eq:mu3lower}, so upon the substitution into
\eqref{eq:M3<} it gives the contribution of the order of $|n|^{5/3}$
into the upper bound for $M_3$.

Next, using the expansion \eqref{eq:cumulants} with $q=1$ and $q=2$,
we obtain
\begin{align}
\notag \sum_{x\in\calX}
x_1^3\mypp\varkappa_1[\nu(x)]\myp\varkappa_2[\nu(x)]&=
\sum_{x\in\calX} x_1^3\sum_{k=1}^\infty k
a_k\mypp\rme^{-k\langle\alpha,\myp x\rangle}\sum_{\ell=1}^\infty
\ell^2 a_\ell\,\rme^{-\ell\myp\langle\alpha,\myp x\rangle}\\
\label{eq:x1^3}
&\le \sum_{k,\myp \ell\myp\ge1} k|a_k|\mypp
\ell^2|a_\ell|\sum_{x\in\ZZ^2_+} x_1^3
\,\rme^{-(k+\ell)\myp\langle\alpha,\myp x\rangle}.
\end{align}
Using the notation \eqref{eq:q} and the bounds of Lemmas
\ref{lm:C_k} and~\ref{lm:exp<C}, the internal sum in \eqref{eq:x1^3}
can be estimated, uniformly in $k,\ell\ge1$, as follows (cf.\
\eqref{eq:2,1})
\begin{align}
\notag
\sum_{x\in\ZZ^2_+} x_1^3\mypp
\rme^{-(k+\ell)\myp\langle\alpha,\myp x\rangle}&=\sum_{x_1=1}^\infty
x_1^3\mypp \rme^{-(k+\ell)\myp\alpha_1 x_1}
\sum_{x_2=0}^\infty\rme^{-(k+\ell)\alpha_2 x_2}\\
\notag &=S_4((k+\ell)\myp\alpha_1)\cdot
\frac{1}{1-\rme^{-(k+\ell)\myp\alpha_2}}\\
\notag &\le \frac{\bar{c}_4\mypp\rme^{-(k+\ell)\myp\alpha_1} }{(1-\rme^{-(k+\ell)\myp\alpha_1})^4
\mypp(1-\rme^{-(k+\ell)\myp\alpha_2})}\\
\label{eq:(k+l)^5}
&=\frac{O(1)}{(k+\ell)^5\myp\alpha_1^4\alpha_2}=\frac{O(|n|^{5/3})}{(k+\ell)^5},
\end{align}
%according to Lemma~\ref{lm:C_k} and the parameterization \eqref{alpha}.
in view of the asymptotics $\alpha_1\asymp \alpha_2\asymp
|n|^{-1/3}$ (see~\eqref{alpha}). The analogous sum with $x_2^3$ in
place of $x_1^3$ in \eqref{eq:x1^3} is estimated similarly, so
combining \eqref{eq:M3<} and \eqref{eq:(k+l)^5} we get
\begin{equation}\label{eq:M3<)k+l)5}
M_3=O(|n|^{5/3}) \sum_{k,\myp \ell\myp\ge1} \frac{k|a_k|\mypp
\ell^2|a_\ell|}{(k+\ell)^5}\myp.
\end{equation}
Furthermore, by the elementary inequality
\begin{equation*}
%\label{eq:k+m>}
(k+\ell)^5=(k+\ell)^2(k+\ell)^3\ge k^{2}\ell^{3}
\end{equation*}
the (double) series on the right-hand side of \eqref{eq:M3<)k+l)5}
is bounded by
\begin{equation*}
\sum_{k,\myp \ell\ge 1} \frac{k\myp |a_k| \:\ell^2
|a_\ell|}{k^{2}\myp \ell^{3}}= \sum_{k=1}^\infty \frac{|a_k|}{k}
\sum_{\ell=1}^\infty \frac{|a_\ell|}{\ell}=A^+(1)^2 <\infty,
\end{equation*}
according to the lemma's hypothesis.
%and the asymptotics
%$\alpha_1\asymp \alpha_2\asymp |n|^{-1/3}$ (see~\eqref{alpha}).
Thus, returning to \eqref{eq:M3<)k+l)5} we see that
$M_3=O(|n|^{5/3})$, and together with the lower bound
\eqref{eq:M-lower} this completes the proof of~\eqref{eq:sum(mu3)}.
\end{proof}

\section{A local limit theorem and the limit shape}\label{sec5}

\subsection{Statement of the theorem}\label{sec5.1}

The role of the local limit theorem in our approach is to yield the
asymptotics of the probability $\QQ_z\{\xi=n\}\equiv\QQ_z(\CP_n)$
appearing in the representation of the measure $\PP_n$ as a
conditional distribution, $\PP_n(\cdot)=\QQ_z(\cdot\mypp|\CP_n)
=\QQ_z(\cdot)/\QQ_z(\CP_n)$ (see Section~\ref{sec2.1}).

To prove such a theorem (see Theorem~\ref{th:LCLT} below), we will
require a technical condition on the generating function $\beta(u)$
as follows.

\begin{assumption}\label{as:7.1}
There exists a constant $\delta_*\myn>0$ such that for any
$\theta\in(0,1)$ the function $u\mapsto \ln\beta(u)$ \myp($u\in\CC$)
satisfies the inequality
\begin{equation}\label{eq:<a1}
\ln\beta(\theta) -\Re\myp(\ln\beta(\theta\mypp\rme^{\myp\rmi t}))\ge
\delta_* \mypp\theta\mypp(1-\cos t),\qquad t\in\RR\myp.
\end{equation}
\end{assumption}

\begin{remark}\label{rm:a1>0}
In terms of the coefficients $\{a_k\}$ in the power series expansion
of the function $\ln\beta(u)$ (see~\eqref{eq:ln-beta}), the
left-hand side of \eqref{eq:<a1} is expressed as
$\sum_{k=1}^\infty\myn a_k\myp \theta^k(1-\cos kt)$. Consequently,
if $a_1>0$ and $a_k\ge0$ for all $k\ge2$ then the inequality
\eqref{eq:<a1} is satisfied, with $\delta_*=a_1>0$.
\end{remark}

As before, we denote $\mu_z=\EE_{z}(\xi)$, \myp$K_z=
\Cov_z(\xi,\xi)$, \myp$V_z=K_z^{-1/2}$ (see Section~\ref{sec4.1}).
Consider the probability density function of a two-dimensional
normal distribution $\mathcal{N}(\mu_z,K_z)$ \,(i.e., with mean
$\mu_z$ and covariance matrix $K_z$), given by
\begin{equation}\label{eq:phi1}
f_{\mu_z\myn,\myp K_z}(x)= \frac{1}{2\pi\sqrt{\det
K_z\vphantom{^k}}}\,\exp\left(-{\textstyle\frac12}|(x-\mu_z)\myp
V_z|^2\right),\qquad x\in\RR^2.
\end{equation}

\begin{theorem}\label{th:LCLT}
Assume that $A^+(1)<\infty$ and suppose that Assumption\/
\textup{\ref{as:7.1}} holds. Then, uniformly in\/ $x\in\ZZ^2_+$\myp,
\begin{equation}\label{eq:LCLT}
\QQ_{z}\{\xi=x\}=f_{\mu_z\myn,\myp K_z}(x)+O(|n|^{-5/3}),\qquad
n\to\infty.
\end{equation}
\end{theorem}

%\begin{remark}
%The condition $a_1>0$, equivalent to $b_1=\beta'(0)>0$, ensures that
%the distribution of $\nu(x)$ ($x\in\mathcal{X}$) has a unit maximal
%step, which, together with the condition $b_0=1>0$ implies that the
%distribution of $\xi=\sum_{x\in \mathcal{X}}x\nu(x)$ is supported on
%the entire set $\ZZ_+^2$\myp, that is, $\QQ_z\{\xi=m\}>0$ for any
%$m\in \ZZ_+^2$\myp. ????
%\end{remark}

\begin{corollary}\label{cor:Q}
Under the conditions of\/ Theorem\/ \textup{\ref{th:LCLT}}
\begin{equation}\label{sim}
\QQ_{z}\{\xi=n\}\asymp |n|^{-4/3}, \qquad n\to\infty.
\end{equation}
\end{corollary}

With the asymptotic results of Sections \ref{sec3.3} and
\ref{sec4.1} at hand, it is not difficult to deduce the corollary
from the theorem.

\begin{proof}[Proof\/ of\/ Corollary\/ \textup{\ref{cor:Q}}] According to
Theorem~\ref{th:4.1}, we have $\mu_z =n+O(|n|^{2/3})$. Together with
the asymptotics of $\|V_z\|$ (see~\eqref{eq:KV}) this implies
\begin{align*}
|(n-\mu_z)\myp V_z|&\le |\myp n-\mu_z|\,{\cdot}\, \|V_z\|
=O(|n|^{2/3})\cdot |n|^{-2/3}=O(1).
\end{align*}
Hence, with the help of Lemma \ref{lm:detK} we get
\begin{align*}
f_{\mu_z\myn,\myp K_z}(n)&=\frac{1}{2\pi\sqrt{\det
K_z\vphantom{^k}}} \:\rme^{-|(n-\mu_z)V_z|^2\myn/2}\asymp
\frac{1}{\sqrt{\det K_z\vphantom{^k}}}\asymp |n|^{-4/3},
\end{align*}
and \eqref{sim} now readily follows from~\eqref{eq:LCLT}.
\end{proof}

\subsection{Estimates of the characteristic functions}\label{sec5.2}

Before proving Theorem \ref{th:LCLT}, we have to make some technical
preparations. Recall from Section~\ref{sec2.1} that, with respect to
the measure $\QQ_{z}$, the random variables $\{\nu(x)\}_{x\in
\calX}$ are independent and have the characteristic
functions~\eqref{eq:c.f.}. Hence, the characteristic function
$\varphi_{\xi}(\lambda):=\EE_{z}
(\rme^{\myp\rmi\langle\lambda,\,\xi\rangle})$ of the vector sum
$\xi=\sum_{x\in \calX} x\myp\nu(x)$ is given by
\begin{equation}\label{x.f_5_0}
\varphi_{\xi}(\lambda)=\prod_{x\in \calX} \varphi_{\nu(x)}(\langle
\lambda,x\rangle)=\prod_{x\in \calX} \frac{\beta(z^{x}
\rme^{\myp\rmi \langle \lambda,\myp
x\rangle})}{\beta(z^{x})}\myp,\qquad \lambda\in\RR^2.
\end{equation}

The next lemma provides a useful estimate (proved in
\cite[Lemma~7.12]{BZ4}) for the characteristic function
$\varphi_{\xi^0}(\lambda)=\rme^{-\langle\lambda,\myp
\mu_z\myn\rangle}\varphi_{\xi}(\lambda)$ of the centered random
vector $\xi^0\myn:=\xi-\mu_z$\myp. Recall that the Lyapunov ratio
$L_z$ is defined in \eqref{L3}, and that $V_z=K_z^{-1/2}$ (see
Section~\ref{sec4.1}).
\begin{lemma}\label{lm:7.2_F}
If\/ $y\in\RR^2$ is such that $|y|\le L_z^{-1}$ then
\begin{equation*}
%\label{16L}
\bigl|\varphi_{\xi^0}(y\myp V_z)- \rme^{-|y|^2\myn/2}\bigr| \le
16\myp L_z|y|^3\mypp \rme^{-|y|^2/6}.
\end{equation*}
\end{lemma}

Under Assumption \ref{as:7.1}, $\varphi_{\xi}(\lambda)$ admits a
simple global bound (cf.\ \cite[Lemma 7.13]{BZ4}).
\begin{lemma}\label{lm:7.3}
Suppose that Assumption\/ \textup{\ref{as:7.1}} is satisfied
\textup{(}with $\delta_*\myn>0$\textup{)}. Then
\begin{equation}\label{f_J}
|\varphi_{\xi}(\lambda)|\le\exp\{-\delta_*
J_{\alpha}(\lambda)\},\qquad \lambda\in\RR^2,
\end{equation}
where
\begin{equation}\label{J_0}
J_{\alpha}(\lambda):=\sum_{x\in \calX} \rme^{-\langle\alpha,\myp
x\rangle}\bigl(1-\cos\myn\langle\lambda,x\rangle\bigr),\qquad
\lambda\in\RR^2.
\end{equation}
\end{lemma}

\begin{proof}
From \eqref{x.f_5_0} we have
\begin{equation}\label{eq:J_1}
|\varphi_{\xi}(\lambda)| =\exp\Biggl\{\sum_{x\in \calX}
\ln\mynn\bigl|\varphi_{\nu(x)}(\langle\lambda,x\rangle)\bigr|\Biggr\}.
\end{equation}
Furthermore, using \eqref{eq:ln_c.f.} and Assumption \ref{as:7.1}
(with $\theta=z^x$, see \eqref{eq:<a1}), we have
\begin{align*} \ln\mynn|\varphi_{\nu(x)}(t)|
=\Re\myp(\ln\varphi_{\nu(x)}(t)) &=
\Re\myp(\ln \beta(z^x\rme^{\myp\rmi t})) -\ln\beta(z^x)\\[.2pc]
&\le -\delta_* \myp z^{x}\myp(1-\cos t\myp),\qquad t\in\RR\myp.
\end{align*}
Utilizing this estimate under the sum in \eqref{eq:J_1} (with
$t=\langle\lambda,x\rangle$) and substituting $z^x=\rme^{-\langle
\alpha,\myp x\rangle}$ (see the notation~\eqref{alpha}), we arrive
at the inequality~\eqref{f_J}.
\end{proof}

\subsection{Proof of Theorem\/ \textup{\ref{th:LCLT}}}
\label{sec5.3}

By definition, the characteristic function of the random vector
$\xi$ is given by the Fourier series
\begin{equation*}
\varphi_{\xi}(\lambda)
=\sum_{x\in\ZZ_+^2}\QQ_{z}\{\xi=x\}\,\rme^{\myp\rmi\langle
\lambda,\myp{}m\rangle},\qquad \lambda\in\RR^2,
\end{equation*}
hence the Fourier coefficients are expressed as
\begin{equation}\label{l_1}
\QQ_{z}\{\xi=x\}=\frac{1}{4\pi^2}\int_{T^2} \rme^{-\rmi\langle
\lambda,\myp{}x\rangle}\mypp
\varphi_{\xi}(\lambda)\,\dif{}\lambda,\qquad x\in\ZZ_+^2\myp,
\end{equation}
where $T^2:=\{\lambda=(\lambda_1,\lambda_2)\in\RR^2\colon
|\lambda_1|\le\pi,\, |\lambda_2|\le\pi\}$. On the other hand, the
characteristic function corresponding to the normal probability
density $f_{\mu_z\myn,\myp K_z}(\cdot)$ (see~\eqref{eq:phi1}) is
given by
\begin{equation*}
\varphi_{\mu_z\myn,\myp K_z}(\lambda)=
\rme^{\myp\rmi\langle\lambda,\myp{}\mu_z\myn\rangle-|\lambda
V_z^{-1}\myn|{\vphantom{(_z}}^2\myn/2},\qquad \lambda\in\RR^2,
\end{equation*}
so by the Fourier inversion formula
\begin{equation}\label{f_o}
f_{\mu_z\myn,\myp K_z}(x)= \frac{1}{4\pi^2}
\int_{\RR^2}\rme^{-\rmi\langle\lambda,\myp
x\rangle}\mypp\rme^{\myp\rmi\langle\lambda,\myp
\mu_z\myn\rangle-|\lambda
V_z^{-1}\myn|{\vphantom{(_z}}^2\myn/2}\,\dif{}\lambda\myp,\qquad
x\in\ZZ^2_+\myp.
\end{equation}

Denote $D_z:=\{\lambda\in\RR^2\colon |\lambda V_z^{-1}\myn|>
L_z^{-1}\}$. If $\lambda\in D_z^c:=\RR^2\setminus D_z$ then, on
account of the asymptotics of $\|V_z\|$ and $L_z$ (see \eqref{eq:KV}
and \eqref{eq:L}, respectively), we get
\begin{equation*}
%\label{VL}
|\lambda|=|\lambda V_z^{-1} V_z|\le |\lambda
V_z^{-1}\myn|\cdot\|V_z\| \le L_z^{-1} \|V_z\|=O(|n|^{-1/3})=o(1),
\end{equation*}
which implies that $D_z^c\subset T^2$ for all $n=(n_1,n_2)$ large
enough. Hence, subtracting \eqref{f_o} from \eqref{l_1} it is easy
to see that, uniformly in $x\in\ZZ^2_+$\myp,
\begin{equation}\label{I}
\bigl|\QQ_{z}\{\xi=x\}-f_{\mu_z\myn,\myp K_z}(x)\bigr|\le
I_1+I_2+I_3\myp,
\end{equation}
where
\begin{align}
\label{eq:I1} &I_1:=\frac{1}{4\pi^2}\int_{D_z^c}
      \bigl|\varphi_{\xi}(\lambda)-\rme^{\myp\rmi
      \langle \lambda,\myp \mu_z\myn\rangle-|\lambda V_z^{-1}\myn|{\vphantom{(_z}}^2\myn/2}
      \bigr|\,\dif{}\lambda\myp,\\[.2pc]
\label{eq:I2} &I_2:=\frac{1}{4\pi^2}\int_{D_z}
      \rme^{-|\lambda V_z^{-1}\myn|{\vphantom{(_z}}^2\myn/2}\,\dif{}\lambda\myp,\\[.2pc]
\label{eq:I3} &I_3:=\frac{1}{4\pi^2}
      \int_{T^2\cap D_z}
      \!|\varphi_{\xi}(\lambda)|\:\dif{}\lambda\myp.
\end{align}

By the substitution $\lambda=y\myp V_z$, the integral \eqref{eq:I1}
is reduced to
\begin{align}
\notag I_1&=\frac{\det V_z}{4\pi^2} \int_{|y|\le L_z^{-1}}
\bigl|\varphi_{\xi}(y V_z)- \rme^{\myp\rmi
      \langle yV_z\myn,\myp \mu_z\myn\rangle-|y|^2\myn/2}\bigr|\,\dif{}y\\
\notag &=\frac{1}{4\pi^2\sqrt{\det K_z}} \int_{|y|\le L_z^{-1}}
\bigl|\varphi_{\xi^0}(y V_z)-
\rme^{-|y|^2\myn/2}\bigr|\,\dif{}y\\[.2pc]
\label{I1} &=O(|n|^{-4/3}) \myp L_z\int_{\RR^2}
|y|^3\mypp\rme^{-|y|^2\myn/6}\,\dif{}y= O(|n|^{-5/3}),
\end{align}
on account of Lemmas \ref{lm:detK}, \ref{lm:7.1} and \ref{lm:7.2_F}.
Similarly, using the change of variables $\lambda=y\myp V_z$ in the
integral \eqref{eq:I2} and passing to the polar coordinates, by
Lemmas \ref{lm:detK} and~\ref{lm:7.1} we get
\begin{align}
\notag I_2&= \frac{\det V_z}{4\pi^2} \int_{|y|>L_z^{-1}}
\!\rme^{-|y|^2\myn/2}\,\dif{y}\\[.2pc]
\label{I2} &= \frac{\det V_z}{2\pi} \int_{L_z^{-1}}^\infty r\mypp
\rme^{-r^2\myn/2}\,\dif{}r=O(|n|^{-4/3})\,
\rme^{-L_z^{-2}\myn/2}=o(|n|^{-5/3}).
\end{align}

Finally, let us turn to the integral \eqref{eq:I3}. Note that if
$\lambda\in\D_z$ (i.e., $|\lambda V_z^{-1}\myn|>L_z^{-1}$), then
$|\lambda|>\eta\myp|\alpha|$ for a small enough constant $\eta>0$,
and hence $\max\{|\lambda_1|/\alpha_1,|\lambda_2|/\alpha_2\}>\eta$;
for otherwise, from \eqref{alpha} and Lemmas \ref{lm:K_z}
and~\ref{lm:7.1} it would follow
\begin{align*}
1<L_z|\lambda V_z^{-1}\myn|&\le L_z\myp\eta\myp|\alpha|\,{\cdot}\,
\|K_z\|^{1/2}= O(\eta)\to0\quad\text{as}\ \ \eta\downarrow0,
\end{align*}
which is a contradiction.  Thus, also using Lemma \ref{lm:7.3} to
estimate the integrand in \eqref{eq:I3}, we get the bound
\begin{align}\label{I3}
I_3&\le \frac{1}{4\pi^2}\sum_{j=1}^2 \int_{T^2}
\mynn\bfOne_{\{|\lambda_j|>\eta\alpha_j\}}(\lambda)\,\rme^{-\delta_*
J_{\alpha}(\lambda)}\,\dif{}\lambda\myp,
\end{align}
where $\bfOne_{B}(\lambda)$ is the indicator of a set
$B\subset\RR^2$. To estimate the first integral in \eqref{I3} (i.e.,
with $j=1$), let us keep in the summation \eqref{J_0} only the pairs
of the form $x=(x_1,1)$, \,$x_1\in\ZZ_+$\myp, giving a lower bound
\begin{align}
\notag J_{\alpha}(\lambda) \ge \sum_{x_1=0}^\infty \rme^{-\alpha_1
x_1}\!\left(1-\Re\myp(\rme^{\myp\rmi (\lambda_1
x_1+\lambda_2)})\right) &=\frac{1}{1-\rme^{-\alpha_1}}-
\Re\left(\frac{\rme^{\myp\rmi\lambda_2}}{1-\rme^{- \alpha_1+\rmi\lambda_1}}\right)\\
\label{J_1}& \ge
\frac{1}{1-\rme^{-\alpha_1}}-\frac{1}{|1-\rme^{-\alpha_1 +
\rmi\lambda_1}|}\myp,
\end{align}
because $\Re\myp(s)\le |s|$ for any $s\in\CC$\myp. Since
$\eta\mypp\alpha_1\le|\lambda_1|\le \pi$, we have
\begin{align*}
|1-\rme^{-\alpha_1+\rmi\lambda_1}|
&\ge|1-\rme^{-\alpha_1+\rmi\eta\myp\alpha_1}|\sim\alpha_1
\sqrt{1+\eta^2} \qquad (\alpha_1\to0).
\end{align*}
Substituting this estimate into \eqref{J_1}, we conclude that
$J_\alpha(\lambda)$ is asymptotically bounded below by
$C(\eta)\mypp\alpha_1^{-1}\mynn \asymp |n|^{1/3}$ (with
$C(\eta):=1-(1+\eta^2)^{-1/2}>0$), uniformly in $\lambda$ such that
$\eta\mypp\alpha_1\le|\lambda_1|\le \pi$. Thus, the first integral
in \eqref{I3} is bounded by
\begin{equation*}
O(1)\exp\bigl(-\myp\const \cdot |n|^{1/3}\bigr)=o(|n|^{-5/3}).
\end{equation*}
The second integral in \eqref{I3} (with $j=2$) is estimated in a
similar fashion by reducing the summation in \eqref{J_0} to that
over the pairs $x=(1,x_2)$ only.

As a result, we get that $I_3=o(|n|^{-5/3})$. Substituting this
estimate, together with \eqref{I1} and \eqref{I2}, into \eqref{I} we
obtain \eqref{eq:LCLT}, and the proof of Theorem \ref{th:LCLT} is
complete.

\subsection{Proof of the limit shape results}\label{sec5.4}

With all preparations at hand, we are finally in a position to prove
the uniform convergence of the scaled polygonal paths
$\tilde\xi_n(\cdot):=(n_1^{-1}\xi_1(\cdot),\myp
n_2^{-1}\xi_2(\cdot))$ to the limit
$g^*(\cdot)=(g_1^*(\cdot),g_2^*(\cdot))$ \emph{in probability} with
respect to both measures $\QQ_z$ and $\PP_n$.
% (see the notation defined in the Introduction).
Note that Theorems \ref{th:LSQ} and \ref{th:LSP} below can be easily
reformulated using the tangential distance $d_{\mathcal
T}(\tilde{\varGamma}_n,\gamma^*)$ defined in \eqref{eq:d-T} (cf.\
Theorem \ref{th:main1} which is stated in these terms).

Let us first establish the universality of the limit shape under the
measure $\QQ_z$.
\begin{theorem}\label{th:LSQ}
Under Assumption \textup{\ref{as:z}}, for each\/ $\varepsilon>0$ we
have
\begin{equation*}
\lim_{n\to\infty} \QQ_{z}\biggl\{\sup_{0\le t\le\infty}\bigl|\myp
n_j^{-1}\xi_j(t)-g^*_j(t)\bigr| \le\varepsilon\biggr\}=1\qquad
(j=1,2).
\end{equation*}
\end{theorem}

\begin{proof}
By Theorems \ref{th:3.2} and \ref{th:8.1.1a}, the expectation of the
random process $n_j^{-1}\xi_j(t)$ uniformly converges to $g^*_j(t)$
as $n\to\infty$. Therefore, we only need to check that, for each
$\varepsilon>0$,
\begin{equation}\label{eq:lim-Q}
\lim_{n\to\infty}\QQ_{z}\biggl\{\sup_{0\le t\le\infty}
n_j^{-1}\bigl|\myp\xi_j(t)-\EE_{z}[\myp\xi_j(t)]\bigr|>\varepsilon\biggr\}=0.
\end{equation}

Note that the random process
$\xi^0_{j}(t):=\xi_j(t)-\EE_{z}[\myp\xi_j(t)]$ has independent
increments and zero mean, hence it is a martingale with respect to
the filtration ${\mathcal F}_t=\sigma\{\nu(x),\, x\in \calX(t)\}$,
\,$t\in[0,\infty]$. From the definition of $\xi_j(t)$
(see~\eqref{xi(t)}), it is also clear that $\xi^0_{j}(\cdot)$ is
\emph{c\`adl\`ag} (i.e., its paths are everywhere right-continuous
and have left limits). Therefore, applying the Doob--Kolmogorov
submartingale inequality (see, e.g., \cite[Theorem~6.14,
p.~99]{Yeh})
%(see, e.g., \cite[Ch.~II, Theorem~1.7, p.~54]{RevYor})
and using Theorem \ref{th:K}, we obtain
\begin{align*}
&\QQ_{z}\biggl\{\sup_{0\le t\le\infty}\myn|\myp\xi^0_{j}(t)|>
\varepsilon\myp n_j\biggr\}
\le\frac{\Var_z(\xi_j(\infty))}{\varepsilon^2\myp n_j^2}\asymp
|n|^{-2/3}\to0,\qquad n\to\infty.
\end{align*}
%recalling that $\xi_j\equiv \xi_j(\infty)$ and therefore
%$\xi^0_{j}(\infty)\equiv \xi^0_{j}=\xi_j-\EE_{z}(\xi_j)$ (see the
%notation in Lemma~\ref{lm:6.7}).
Hence, the limit \eqref{eq:lim-Q} follows.
\end{proof}

Let us now prove our main result about the universality of the limit
shape under the measure $\PP_n$ (cf.\ Theorem~\ref{th:main1}).

\begin{theorem}\label{th:LSP}
Let $A^+(1)<\infty$ and Assumption \textup{\ref{as:7.1}} be
satisfied. Then for any\/ $\varepsilon>0$
\begin{equation*}
\lim_{n\to\infty} \PP_n\biggl\{\sup_{0\le t\le\infty}\bigl|\myp
n_j^{-1}\xi_j(t)-g^*_j(t)\bigr| \le\varepsilon\biggr\}=1\qquad
(j=1,2).
\end{equation*}
\end{theorem}

\begin{proof}
Like in the proof of Theorem \ref{th:LSQ}, the claim is reduced to
the limit (cf.~\eqref{eq:lim-Q})
\begin{equation}\label{A1-0}
\lim_{n\to\infty} \PP_n\biggl\{\sup_{0\le t\le\infty}
\mynn|\myp\xi^0_{j}(t)|>\varepsilon\myp n_j\biggr\}=0,
\end{equation}
where $\xi^0_{j}(t)=\xi_j(t)-\EE_{z}[\myp\xi_j(t)]$.
% (with $\xi^0_{j}(\infty)\equiv \xi^0_{j}=\xi_j-\EE_{z}[\myp\xi_j]$, see
%the notation in Lemma~\ref{lm:6.7}) .
Using the definition \eqref{Pn} we easily get the bound
\begin{align}\label{A1}
\PP_n\biggl\{\sup_{0\le t\le\infty}\mynn|\myp\xi^0_{j}(t)|>
\varepsilon\myp n_j\biggr\} \le \frac{\QQ_{z}\!\left\{\sup_{0\le
t\le\infty}\myn|\myp\xi^0_{j}(t)|> \varepsilon\myp
n_j\right\}}{\QQ_{z}\{\xi=n\}}\myp.
\end{align}
Again applying the Doob--Kolmogorov submartingale inequality
\cite[Theorem~6.14, p.~99]{Yeh} (but now with the sixth moment) and
using Lemma \ref{lm:6.7}, we obtain
\begin{equation}\label{A1-1}
\QQ_{z}\biggl\{\sup_{0\le
t\le\infty}\mynn|\myp\xi^0_{j}(t)|>\varepsilon\myp n_j \biggr\}\le
\frac{\EE_{z}\myn\bigl[(\xi_j^0)^{6}\bigr]}{\varepsilon^6 \myp
n_j^{6}} \asymp
%\frac{|n|^4}{|n|^6}=
|n|^{-2}.
\end{equation}
On the other hand, by Corollary \ref{cor:Q}
\begin{equation}\label{A1-2}
\QQ_{z}\{\xi=n\}\asymp |n|^{-4/3}.
\end{equation}
Combining \eqref{A1-1} and \eqref{A1-2}, we conclude that the
right-hand side of \eqref{A1} is dominated by a quantity of order of
$O(|n|^{-2/3})\to0$, and so the limit in \eqref{A1-0} follows.
\end{proof}

\section{Examples}\label{sec6}

%In this section, we introduce six examples by specifying the
%generating function $\beta(u)=\sum_{k=0}^\infty b_k u^k$ and the
%orresponding function $\ln\myn\beta(u)=\sum_{k=1}^\infty a_k u^k$.
%

Let us now consider a few illustrative examples by specifying the
generating function $u\mapsto \beta(u)=\sum_{k=0}^\infty b_k u^k$
(see \eqref{eq:beta}). Although the associated multiplicative
measures $\QQ_z$ and $\PP_n$ are defined primarily in terms of the
coefficients $\{b_k\}$ (see \eqref{Q1} and \eqref{condP},
respectively), explicit expressions for $b_k$ may be complicated, so
we will not always attempt to give such expressions. For our
purposes, it is more important to focus on the function $u\mapsto
\ln\beta(u)$ and its power expansion coefficients $\{a_k\}$, since
these are the ingredients that determine the convergence to the
limit shape \eqref{eq:gamma0}. In particular, we will have to check
the basic condition $A^+(2)<\infty$ (see Assumption~\ref{as:z}), as
well as the refined condition $A^+(1)<\infty$ and Assumption
\ref{as:7.1}, both needed for the limit shape result under the
measure $\PP_n$ (see Theorem~\ref{th:LSP}).

\begin{remark}
It is worth pointing out that Examples \ref{ex:1}, \ref{ex:2} and
\ref{ex:3} have direct analogs in the theory of decomposable
combinatorial structures, corresponding to the well-known
meta-classes of \emph{multisets}, \emph{selections}, and
\emph{assemblies}, respectively (see \cite{AT}
and~\cite[\S\myp2.2]{ABT}). For further details about this
correspondence and, more generally, for an extensive discussion of
the combinatorial interpretation of the generating functions
described in Examples \ref{ex:1}\myp--\,\ref{ex:6} below, the reader
is referred to the recent paper \cite[\S\S\,6.1,~6.2]{Bogachev}.
%More specifically, Example \ref{ex:1}
%below belongs to the class of weighted partitions, including the
%case of unrestricted partitions under the uniform (equiprobable)
%distribution; Example \ref{ex:2} leads to (weighted) partitions with
%bounds on the part counts, including uniformly distributed strict
%partitions (i.e., with distinct parts); Example \ref{ex:3} includes
%set partitions with labeled elements and ordered contents. Examples
%\ref{ex:4} and \ref{ex:5}, as well as Example \ref{ex:3}, are
%instances of the so-called \emph{exponential structures} (see, e.g.,
%\cite[\S\myp5.5]{Stanley}). To the best of our knowledge, the
%generating function of Example \ref{ex:6} appears to be new: it was
%first considered by the author \cite{Bogachev} in the context of
%one-dimensional integer partitions.
\end{remark}

\begin{example}
%[multisets]
\label{ex:1} For $r\in(0,\infty)$, $\rho\in(0,1]$, let $\QQ_{z}$ be
a measure on the space $\CP$ determined by the formula \eqref{Q}
with coefficients
\begin{equation*}
b_k=\binom{r+k-1}{k}\, \rho^k,\qquad k\in\ZZ_+\myp.
\end{equation*}
A particular case with $\rho=1$ was considered in \cite{BZ4}
(cf.~\eqref{eq:b-k-r}). Note that $b_0=1$, in accordance with our
convention in Section~\ref{sec2.1}, and $b_1=r\rho$. By the binomial
expansion formula, the generating function of the sequence
\eqref{eq:b-k-r} is given by
\begin{equation}\label{G1}
\beta(u)=(1-\rho u)^{-r}, \qquad |u|<\rho^{-1},
\end{equation}
and formula \eqref{Q} specializes to
\begin{equation}\label{Qb1}
\QQ_{z}\{\nu(x)=k\}=\binom{r+k-1}{k}\, \rho^k z^{k x}(1-\rho
z^x)^r,\ \quad k\in\ZZ_+\ \quad (x\in\calX),
\end{equation}
which is a negative binomial distribution with parameters $r$ and
$p=1-\rho z^x$.

If $r=1$ then $b_k=\rho^k$, \,$\beta(u)=(1-\rho u)^{-1}$ and,
according to~\eqref{Qb1},
\begin{equation*}
%\label{Qb1-geom}
\QQ_{z}\{\nu(x)=k\}= \rho^k z^{k x}(1-\rho z^x),\ \quad k\in\ZZ_+\
\quad (x\in\calX).
\end{equation*}
In turn, from formulas \eqref{eq:b-Gamma} and \eqref{condP} we get
\begin{equation}\label{eq:r=1}
\PP_n(\varGamma)=\frac{\rho^{N_{\varGamma}}}{\sum_{\varGamma'\in\CP_n}
\rho^{N_{\varGamma'}}},\qquad \varGamma\in\CP_n,
\end{equation}
where $N_{\varGamma}:= \sum_{x\in \calX} \nu(x)$ is the total number
of integer points on $\varGamma\setminus\{0\}$. Furthermore, if also
$\rho=1$ then \eqref{eq:r=1} is reduced to the uniform distribution
on $\CP_n$ (see~\eqref{condP}),
$$
\PP_n(\varGamma)=\frac{1}{\#(\CP_n)}\myp,\qquad \varGamma\in\CP_n.
$$

In the general case, using \eqref{G1} we note that
$$
\ln \beta(u)=-r\ln\myn(1-\rho u)=r\sum_{k=1}^\infty\frac{\rho^k
u^k}{k}\myp,
$$
and so the coefficients $\{a_k\}$ in the expansion
\eqref{eq:ln-beta} are given by
$$
a_k=\frac{r\rho^k}{k}>0,\qquad k\in\NN\qquad (0<\rho\le 1).
$$
As pointed out in Remark \ref{rm:a1>0}, this implies that Assumption
\ref{as:7.1} is satisfied; also, it readily follows that
$A^+(\sigma)<\infty$ for any $\sigma>0$ (and each $\rho\in(0,1]$).
\end{example}

\begin{example}
%[selections]
\label{ex:2} For \,$m\in\NN$, \,$\rho\in(0,1]$, consider the
generating function
\begin{equation}\label{eq:f(s)2}
\beta(u)=(1+\rho u)^m,\qquad |u|<\rho^{-1},
\end{equation}
with the coefficients in the expansion \eqref{eq:beta} given by
\begin{equation*}
%\label{b_k2}
b_k=\binom{m}{k}\,\rho^k=\frac{m\myp(m-1)\cdots(m-k+1)}{k!}\,\rho^k,\qquad
k=0,1,\dots,m.
\end{equation*}
In particular, $b_0=1$, $b_1=m\rho$. Accordingly, the formula
\eqref{Q} gives a binomial distribution
\begin{equation}\label{Qb2}
\QQ_{z}\{\nu(x)=k\}=\binom{m}{k} \frac{\rho^k z^{k x}}{(1+\rho
z^x)^{m}}\myp,\ \quad k=0,1,\dots,m\ \quad (x\in\calX),
\end{equation}
with parameters $m$ and $p=\rho z^x(1+\rho z^x)^{-1}$.

In the special case $m=1$, the measure $\QQ_z$ is concentrated on
the subspace $\check{\CP}$ of polygonal lines with ``simple'' edges,
that is, containing no lattice points between the adjacent vertices.
Here we have $b_0=1$, $b_1=\rho$, and $b_k=0$ \,($k\ge2$), so that
\eqref{Qb2} is reduced to
\begin{equation*}
\QQ_{z}\{\nu(x)=k\}= \frac{\rho^k z^{k x}}{1+\rho z^x}\myp,\ \quad
k=0,1\ \quad (x\in\calX).
\end{equation*}
Accordingly, the formula \eqref{condP} specifies on the
corresponding subspace $\check{\CP}_n$ the distribution
\begin{equation}\label{eq:r=1-2}
\PP_n(\varGamma)=\frac{\rho^{N_{\varGamma}}}{\sum_{\varGamma'\in\check\CP_n}
\rho^{N_{\varGamma'}}}\myp,\qquad \varGamma\in \check{\CP}_n,
\end{equation}
where the number of integer points $N_\varGamma$ coincides here with
the number of vertices on $\varGamma\setminus\{0\}$. Furthermore, if
also $\rho=1$ then \eqref{eq:r=1-2} is reduced to the uniform
distribution on $\check{\CP}_n$,
\begin{equation*}
%\label{eq:r=1-2-}
\PP_n(\varGamma)=\frac{1}{\# (\check{\CP}_{n})}\myp,\qquad
\varGamma\in \check{\CP}_n.
\end{equation*}

In the general case, from \eqref{eq:f(s)2} we obtain
\begin{equation}\label{eq:ln(beta)2}
\ln \beta(u)=m\myp\ln\myn(1+\rho u)=m\sum_{k=1}^\infty
\frac{(-1)^{k-1}\rho^k}{k}\,u^k,
\end{equation}
hence the coefficients $\{a_k\}$ in the expansion \eqref{eq:ln-beta}
are given by
$$
a_k=\frac{m\myp(-1)^{k-1}\rho^k}{k},\qquad
k\in\NN\qquad(0<\rho\le1),
$$
and in particular $a_1=m\rho>0$. Note that $A^+(\sigma)<\infty$ for
any $\sigma>0$. Finally, let us check that Assumption \ref{as:7.1}
holds.
%It is more convenient to use the version
%\eqref{eq:<a1-1}.
Using \eqref{eq:ln(beta)2} we obtain, for any $\theta\in(0,1)$ and
all $t\in\RR$\myp,
\begin{align*}
\ln\beta(\theta)-\Re\myp(\ln\beta(\theta\mypp\rme^{\rmi t}))
&=m\myp\ln\myn(1+\rho \mypp\theta)-m\mypp\Re\myp\bigl(\ln\myn(1+\rho
\mypp\theta\mypp\rme^{\rmi t})\bigr)\\
&=m\myp\ln\myn(1+\rho\mypp\theta)-m\myp\ln\myn|1+\rho\mypp\theta\mypp\rme^{\rmi t}|\\
&=-\frac{m}{2}\ln\!\myn\left(\frac{1+2\rho\mypp\theta\cos t+\rho^2\theta^2}{(1+\rho\mypp\theta)^2}\right)\\
&\ge -\frac{m}{2}\left(\frac{1+2\rho\mypp\theta\cos t+
\rho^2\theta^2}{(1+\rho\mypp\theta)^2}-1\right)\\
&= \frac{m\rho\,\theta\mypp(1-\cos t)}{(1+\rho\mypp\theta)^2}\ge
\frac{m\rho}{(1+\rho)^2}\,\theta\mypp(1-\cos t)\myp.
\end{align*}
Thus, the inequality \eqref{eq:<a1} holds with
$\delta_*=m\rho/(1+\rho)^2>0$.
% Substituting \eqref{eq:f(s)2} and
%recalling that $b_1=r\rho>0$, we obtain
%\begin{align*}
%frac12\ln\left(\frac{\beta(\theta\mypp\rme^{\myp\rmi
%})\myp\beta(\theta\mypp\rme^{-\rmi t})}{\beta(\theta)^2}\right)
%&=\frac{r}{2}\ln\left(\frac{1+2\rho\myp\theta\cos t+\rho^2\theta^2}{(1+\rho\myp\theta)^2}\right)\\
%\le \frac{r}{2}\left(\frac{1+2\rho\myp\theta\cos t+
%\rho^2\theta^2}{(1+\rho\myp\theta)^2}-1\right) \le -\frac{b_1
%theta\mypp(1-\cos t)}{(1+\rho)^2}.
%\end{align*}
\end{example}

\begin{example}
%[assemblies]
\label{ex:3} For $b\in(0,\infty)$, \,$\rho\in[0,1]$, consider the
generating function
\begin{equation}\label{eq:beta3}
\beta(u)=\exp\left(\frac{b\myp u}{1-\rho
u}\right)=\exp\left(b\sum_{k=1}^\infty u^{k}\rho^{k-1}\right),\qquad
|u|<\rho^{-1}.
\end{equation}
Clearly, the corresponding coefficients $b_k$ in the expansion
\eqref{eq:beta} are positive, with $b_0=1$, \,$b_1=b$,
\,$b_2=\frac12\myp b^2+b\myp\rho$, etc. More systematically, one can
use the well-known Fa\`a di Bruno's formula
%generalizing the chain rule to higher derivatives
(see, e.g., \cite[Ch.~I, \S12, p.~34]{Jordan}) to obtain (for
$\rho>0$)
\begin{equation}\label{eq:Faa}
b_k=\rho^k\sum_{m=1}^k\left(\frac{\,b\,}{\rho}\right)^m\!\!\sum_{(j_1\myn,\dots,\myp
j_k) \myp\in\myp\mathcal{J}_m}\frac{1}{j_1\myn!\cdots j_k !}\myp,
\qquad k\in\NN,
\end{equation}
where $\mathcal{J}_m$ is the set of all nonnegative integer
$k$-tuples $(j_1,\dots,j_k)$ such that $j_1+\dots+j_k=m$ and $1\cdot
j_1+2\cdot j_2+\dots+k\cdot j_k =k$.
%\begin{remark}
%Note that the $k$-tuples $(j_1,\dots,j_k)\in\mathcal{J}_m$ are in a
%one-to-one correspondence with partitions of $k$ involving precisely
%$m$ different integers as parts, where an element $j_i$ has the
%meaning of the multiplicity of $i\in\NN$ (i.e., the number of times
%$i$ is used in a partition of $k$).
%\end{remark}

Taking the logarithm of \eqref{eq:beta3}, we see that the
coefficients $\{a_k\}$ in \eqref{eq:ln-beta} are given by
\begin{equation}\label{eq:ak}
a_k=b\myp\rho^{k-1}>0,\qquad k\in\NN\qquad (0<\rho\le1).
\end{equation}
Therefore, Assumption \ref{as:7.1} is automatic (see
Remark~\ref{rm:a1>0}); moreover, $A^+(\sigma)<\infty$ for any
$\sigma>0$, except for the case $\rho=1$ where $A^+(\sigma)<\infty$
only for $\sigma>1$.

In the special case $\rho=0$, we have $\beta(u)=\rme^{\myp bu}$ and
the expression \eqref{eq:Faa} is reduced to $b_k=b^k\myn/k!$\myp,
whereas \eqref{eq:ak} simplifies to $a_1=b$ and $a_k=0$ for $k\ge2$.
In this case, the random variables $\nu(x)$ ($x\in\calX$) have a
Poisson distribution with parameter $b\myp z^x$,
\begin{equation*}
%\label{Qb3}
\QQ_{z}\{\nu(x)=k\}=\frac{b^k z^{k x}}{k!}\,\rme^{-b\myp z^x},\
\quad k \in\ZZ_+\ \quad (x\in\calX),
\end{equation*}
which leads, according to \eqref{condP}, to the following
distribution on $\CP_n$
\begin{equation*}
\PP_n(\varGamma)=\left(\sum_{\{k^{\myp\prime}_x\}\in\CP_n}
\prod_{x\in\calX}\frac{b^{\myp
k^{\myp\prime}_x}}{k^{\myp\prime}_x!}\right)^{-1}
\prod_{x\in\calX}\frac{b^{\myp k_x}}{k_x!}\myp,\qquad
\varGamma\leftrightarrow\{k_x\}\in \CP_n.
\end{equation*}
\end{example}

\begin{example}\label{ex:4}
Extending Example \ref{ex:3} (for simplicity, with $b=1$), let us
set for $r>0$
%, $r\ne1$
and $\rho\in(0,1]$
\begin{equation}\label{eq:beta4a}
\beta(u):=\exp\!\mynn\left(\frac{u}{(1- \rho\myp
u)^r}\right)\!,\qquad |u|<\rho^{-1}.
\end{equation}
Taking the logarithm of \eqref{eq:beta4a} we get the power series
expansion (cf.~\eqref{G1})
\begin{equation}\label{eq:H04}
\ln \beta(u)=\frac{u}{(1-\rho\myp u)^r}=\sum_{k=1}^\infty
\binom{r+k-2}{k-1}\myp\rho^{k-1} u^{k},
\end{equation}
which has positive coefficients $a_k$ (cf.~\eqref{eq:b-k-r}). Hence,
Assumption \ref{as:7.1} is satisfied by virtue of
Remark~\ref{rm:a1>0}. To check the condition $A^+(\sigma)<\infty$,
observe using Stirling's asymptotic formula for the gamma function
(see \cite[\S12.5, p.\;130]{Cramer}) that
\begin{equation*}
%\label{eq:Stirling}
a_k=\binom{r+k-2}{k-1}\myp\rho^{k-1}
=\frac{\Gamma(k+r-1)}{\Gamma(r)\mypp\Gamma(k)}\mypp\rho^{k-1}\sim
\frac{k^{\myp r-1}}{\Gamma(r)}\,\rho^{k-1}\qquad (k\to\infty),
\end{equation*}
hence $A^{+}(\sigma)<\infty$ for any $\sigma>0$ if $\rho<1$, whereas
if $\rho=1$ then $A^{+}(\sigma)<\infty$ only for $\sigma>r$.

On substituting \eqref{eq:H04} into Taylor's expansion of the
exponential function in \eqref{eq:beta4a}, it is evident that the
corresponding coefficients $b_{k}$ in the power series expansion of
$\beta(u)$ are also positive, with $b_0=b_1=1$,
$b_2=r\myn\rho+\frac12$, etc.
%; more generally, $b_k$ can be
%evaluated using Fa\`a di Bruno's formula like in Example \ref{ex:3},
%but we omit the details.
\end{example}

\begin{example}\label{ex:5}
Combining the exponential form of Example \ref{ex:4} with the
generating function from Example \ref{ex:2}, for $\rho\in[0,1]$,
\,$m\in\NN$ consider
\begin{equation}\label{eq:F05}
\beta(u):=\exp\mynn\bigl\{u\myp(1+\rho\myp u)^{m-1}\bigr\}.
%,\qquad u\in\CC.
\end{equation}
Since $u\mapsto u\myp(1+\rho\myp u)^{m-1}$ is a polynomial of degree
$m$ with positive coefficients, it follows that the coefficients
$\{b_k\}$ in the power series expansion of the function
\eqref{eq:F05} are positive for all $k\in\ZZ_+$\myp.

From \eqref{eq:F05} by the binomial formula we obtain the expansion
\begin{equation*}
%\label{eq:H05}
\ln\beta(u)=u\myp(1+\rho\myp
u)^{m-1}=\sum_{k=1}^{m}\binom{m-1}{k-1}\myp\rho^{k-1} u^{k},
\end{equation*}
with the expansion coefficients $a_k>0$ for $k=1,\dots,m$ and
$a_k=0$ for $k\ge m+1$. Hence, Assumption \ref{as:7.1} is satisfied
and $A^+(\sigma)<\infty$ for any $\sigma>0$.
\end{example}

\begin{example}\label{ex:6}
With $r\in(0,\infty)$, \,$\rho\in(0,1]$, consider the generating
function
\begin{equation}\label{eq:f(s)4}
\beta(u)=\left(\frac{-\ln\myn(1-\rho u)}{\rho
u}\right)^r=\left(1+\sum_{k =1}^\infty \frac{\rho^k u^{k }}{k
+1}\right)^r=:\beta_1(u)^r.
\end{equation}
%\sout{From \eqref{eq:f(s)4} it is clear that $b_0=1$,
%$b_1=\frac12\mypp r\rho>0$ and, more generally, all
%$b_k>0$.}
If $r=m\in\NN$ then from \eqref{eq:f(s)4} it is evident that the
coefficients $\{b_k\}$ in the power series expansion of $\beta(u)$
%in \eqref{eq:beta}
are positive for all $k\in\ZZ_+$\myp; however, for non-integer $r>0$
this is not so clear, since the binomial expansion of $(1+t)^r$
involves negative terms. Yet the positivity of $b_k$ for $k\ge0$
holds for \emph{any real} $r>0$, which will be established below.

Let us first analyze the coefficients $\{a_k\}$ in the power series
expansion of $\ln\beta(u)=r\ln\beta_1(u)$ (see~\eqref{eq:f(s)4}).
Differentiation of the identity $r\ln\beta_1(u)=\sum_{k=1}^\infty
a_k u^k$  gives
\begin{equation}\label{eq:dif-once}
r \beta_1'(u)=\beta_1(u)\sum_{k=1}^\infty ka_k\myp u^{k-1}.
\end{equation}
Differentiating \eqref{eq:dif-once} again $k-1$ times ($k\ge1$), by
the Leibniz rule we obtain
\begin{equation}\label{eq:f'(0)}
\beta_1^{(k)}(0)=\frac{1}{r}\sum_{i=0}^{k-1} \binom{k-1}{i}\myp
\beta_1^{(k-1-i)}(0)\, (i+1)! \,a_{i+1},\qquad k\in\NN.
\end{equation}
But we know from \eqref{eq:f(s)4} that $\beta_1^{(j)}(0)=\rho^j
j!/(j+1)$ ($j\in\ZZ_+$), and so the recurrence relation
\eqref{eq:f'(0)} specializes (after some cancellations) to the
equation
\begin{equation}\label{eq:m/m}
\frac{k}{k+1}=\frac{1}{r}\sum_{i=0}^{k-1} \frac{\rho^{-i-1}
(i+1)}{k-i} \,a_{i+1}.
\end{equation}
Furthermore, denoting for short $\tilde{a}_j:=r^{-1}\rho^{-j}j a_j$
($j\in\NN$) we can simplify
%the equation
\eqref{eq:m/m} to
\begin{equation}\label{eq:a_m+1}
\frac{k}{k+1}=\sum_{i=0}^{k-1}\frac{\tilde{a}_{i+1}}{k-i}\myp.
\end{equation}
Setting here $k=1,2,3,\dots$ we find successively
%all $\tilde{a}_k$,
$$
\tilde{a}_1=\tfrac{1}{2},\quad \tilde{a}_2=\tfrac{5}{12},\quad
\tilde{a}_3=\tfrac{3}{8},\quad \tilde{a}_4=\tfrac{251}{720},\ \
\dots
$$
%but the fractions quickly become cumbersome.
More generally, let us prove that
\begin{equation}\label{eq:<a<*}
\frac{1}{k\myp(k+1)}\le \tilde{a}_k\le \frac{k}{k+1}\myp,\qquad
k\in\NN.
\end{equation}
%The proof proceeds by induction.
Since $\tilde{a}_1=\frac12$, the claim \eqref{eq:<a<*} is true for
$k=1$. Suppose now that the inequalities \eqref{eq:<a<*} hold for
$\tilde{a}_1,\dots, \tilde{a}_{k-1}$ ($k\ge2$), which entails that
$\tilde{a}_1,\dots, \tilde{a}_{k-1}>0$. Observe that the recurrence
\eqref{eq:a_m+1} (with $k$ replaced by $k-1$) implies
\begin{align*}
\frac{k}{k+1} &=\sum_{i=0}^{k-2}
\frac{\tilde{a}_{i+1}}{k-i}+\tilde{a}_{k}\le \sum_{i=0}^{k-2}
\frac{\tilde{a}_{i+1}}{k-1-i}+\tilde{a}_{k}=\frac{k-1}{k}+\tilde{a}_{k},
\end{align*}
and it follows that
\begin{equation}\label{eq:a<}
\tilde{a}_{k}\ge
\frac{k}{k+1}-\frac{k-1}{k}=\frac{1}{k\myp(k+1)}\myp.
\end{equation}
On the other hand, using that $\tilde{a}_1,\dots,
\tilde{a}_{k-1}>0$, from \eqref{eq:a_m+1} we also get
\begin{equation}\label{eq:a>}
\frac{k}{k+1}=\tilde{a}_{k}+\sum_{i=0}^{k-2}
\frac{\tilde{a}_{i+1}}{k-i}\ge \tilde{a}_{k}.
\end{equation}
Thus, the inequalities \eqref{eq:a<} and \eqref{eq:a>} prove the
claim \eqref{eq:<a<*} for the $\tilde{a}_k$, and by induction it is
valid for all $k\in\NN$.

For the original coefficients $a_k$, the inequalities
\eqref{eq:<a<*} are rewritten as
\begin{equation}\label{eq:<a<}
\frac{r\rho^k}{k^2\myp(k+1)}\le a_k\le\frac{r\rho^k}{k+1}\myp,\qquad
k\in\NN,
\end{equation}
and in particular $a_k>0$ for all $k\in\NN$, so that Assumption
\ref{as:7.1} is automatically satisfied due to Remark~\ref{rm:a1>0}.
Furthermore, the inequalities \eqref{eq:<a<} imply that
$A^+(\sigma)<\infty$ for any $\sigma>0$.

Finally, we can resolve the question of why the formula
\eqref{eq:f(s)4} defines a generating function with
\emph{nonnegative} coefficients: since Taylor's coefficients of the
exponential function are positive, it is evident from the relation
$\beta(u)=\exp\mynn\left\{\sum_{k=1}^\infty a_k\myp u^k\right\}$
that $b_k>0$ for all $k\in\ZZ_+$\myp.

\end{example}

%\begin{remark}
%Specific choices of the coefficients $(b_k)$ in Examples
%\ref{ex:1}--\ref{ex:4} above can be used in the context of integer
%partitions (see, e.g., \cite{GSE,V3,Yakubovich} and also a recent
%aper \cite{Bogachev}). More specifically, Example \ref{ex:1}
%orresponds to the ensemble of weighted partitions including the case
%f all unrestricted partitions under the uniform distribution;
%Example \ref{ex:2} leads to (weighted) partitions with bounds on the
%multiplicities of parts, including the case of uniform partitions
%ith distinct parts; Example \ref{ex:3} corresponds to partitions
%begin{color}{blue}representing the cycle structure of
%permutations; finally, Example \ref{ex:4} introduced in
%\cite{Bogachev} defines a new ensemble of random partitions. Note
%that the limit shapes of partitions (or rather their Young diagrams)
%in the first three cases are known to exist, at least under some
%technical conditions on the coefficients (see
%\cite{Bogachev,EG,V3,Yakubovich}, but they are all drastically
%different from each other, as opposed to the case of lattice
%polygonal lines representing strict vector partitions, for which the
%limit shape is universal in all four examples.
%\end{remark}

\section*{Acknowledgments}
This work was supported in part by a Leverhulme Research Fellowship.
Partial support by the Hausdorff Research Institute for Mathematics
(Bonn) in the framework of Trimester Program ``Universality and
Homogeneity'' is also acknowledged.
%Partial support by the Hausdorff Research Institute for Mathematics (Bonn) is also acknowledged.
The author is grateful to Boris Granovsky, Ilya
Molchanov, Stanislav Molchanov, Anatoly Vershik and Yuri Yakubovich
for helpful discussions, and to the anonymous referees for pointing
out an error in the original manuscript and for constructive
comments that have helped to improve the presentation.

\end{document}